\documentclass[]{article}

\usepackage{graphicx}
\usepackage{amsmath}
\usepackage{amssymb}
\usepackage{amsthm}
\usepackage{mathrsfs}
\usepackage{pxfonts}
\usepackage{enumerate}
\usepackage{color}
\usepackage{mathdots}
\usepackage{sectsty}
\usepackage{tikz}
\usepackage{adjustbox}
\usepackage{enumitem}
\usepackage{caption}
\usepackage{bbold}
\usepackage[hidelinks]{hyperref}

\usepackage{booktabs}

\allowdisplaybreaks
\sectionfont{\scshape\centering\fontsize{11}{14}\selectfont}
\subsectionfont{\scshape\fontsize{11}{14}\selectfont}
\usepackage{fancyhdr}
\usepackage[nottoc,notlot,notlof]{tocbibind}

\newtheorem{thm}{Theorem}[section]

\newtheorem{lem}[thm]{Lemma}
\newtheorem{defn}[thm]{Definition}
\newtheorem{rmk}[thm]{Remark}
\newtheorem{prop}[thm]{Proposition}

\newtheorem{conj}[thm]{Conjecture}

\newtheorem{ass}[thm]{Assumption}
\newtheorem*{theorem*}{Theorem}

\newcommand\shorttitle{$\mathrm{GL_N}(\mathbb{C})$ Brownian motion and SPDE on entire functions}
\newcommand\authors{T. Assiotis and Z. S. Mirsajjadi}

\title{\large \bf $\mathsf{GL}_N(\mathbb{C})$ BROWNIAN MOTION AND STOCHASTIC PDE ON ENTIRE FUNCTIONS}
\author{\small THEODOROS ASSIOTIS AND ZAHRA SADAT MIRSAJJADI}
\date{}

\begin{document}
\maketitle

\begin{abstract} 
We construct the full edge scaling limit of the singular values of Brownian motion on the general linear group $\mathsf{GL}_N(\mathbb{C})$ starting from general conditions. We show that the limiting paths solve an infinite system of SDE with log-interaction and have a Gibbs resampling property with exponential Brownian bridges. Moreover, we show that the evolution of the limiting rescaled reverse characteristic polynomial solves a stochastic partial differential equation with a non-linear multiplicative noise and linear drift. From a special initial condition the resulting line ensemble coincides, in logarithmic coordinates, with a line ensemble constructed by Ahn which arises as a universal scaling limit of singular values of products of random matrices. We prove some analogous results on the evolution of limiting characteristic polynomials for two models whose stationary measures are given by the Hua-Pickrell and Bessel stochastic zeta functions respectively.
\end{abstract}

\tableofcontents
\addtocontents{toc}{\setcounter{tocdepth}{1}}

\section{Introduction}

The main purpose of this paper two-fold: Firstly, to construct the full edge scaling limit of the singular values of multiplicative Brownian motion on the general linear group $\mathsf{GL}_N(\mathbb{C})$ \cite{NorrisRogersWilliams,Biane,JonesOConnell,IpsenSchomerus} starting from general initial conditions, including a description of the limiting dynamics in terms of an infinite system of stochastic differential equations (SDE) \cite{Tsai,Osada1,Osada2,Osada3,OsadaTanemura} with singular log-interaction and Gibbs resampling properties of the paths \cite{CorwinHammond1,CorwinHammond2,DimitrovCorwinBarraquand,DimitrovMatetski,CalvertHammondHegde,DauvergneVirag,XuanWuBessel}. From the special initial condition being the identity matrix the limiting paths have been studied recently in the literature \cite{Ahn, Mustazee}, giving rise to a universal critical stochastic process that is related to products of random matrices \cite{BougerolLacroix,AkemannKieburgWei,AkemannIpsen,KuijlaarsZhang,KuijlaarsStivigny} when one scales in a balanced way the number of factors in the product and the size of the matrices, see \cite{Akemann2019integrable,AkemannUniversality,Ahn,LiuWangWang,AkemannByun,BerezinStrahovGap}. This critical limit process also arises in certain last passage percolation models, see \cite{BerezinStrahov}. For the analogous and significantly more developed universality picture in the additive case of Dyson Brownian motion \cite{Dyson} in random matrix theory and beyond see \cite{ErdosYau,CorwinKPZ}.

Second, to show that the dynamics of the associated rescaled reverse characteristic polynomials converge to a limiting evolution on entire functions which solves a certain stochastic partial differential equation (spde) with a nonlinear multiplicative noise and a linear drift. We prove similar results for dynamics with invariant measure given by the stochastic zeta functions associated to the Hua-Pickrell and the Bessel point processes \cite{ChhaibiNajnudelNikeghbali,ValkoVirag2,LiValko,Theo-RandomEntireFunctions} respectively. As far as we can tell, these are the first explicit examples in the literature of such stochastic evolutions associated to these random entire functions.

The impetus of our work is some hidden integrability of all of these models, a certain consistency as the matrix size $N$ varies, which allows us to apply a powerful formalism introduced by Borodin and Olshanski \cite{BorodinOlshanskiMarkov,BorodinOlshanskiThoma,OlshanskiLectureNotes,OlshanskiICM} that we have more recently further developed in the context of random matrices \cite{HuaPickrellDiffusions,AM24}.

\subsection{Main results}
 We need to introduce some notation and terminology. First define,
\begin{equation*}
\mathbf{W}_{N;+}\left\{\mathbf{x}=(x_i)_{i=1}^N\in \mathbb{R}^{N}:x_1\ge x_2 \ge \cdots \ge x_N\ge 0\right\},
\end{equation*}
and write $\mathbf{W}_{N;+}^\circ$ for its interior. We also define, endowed with product topology,
\begin{align*}
\mathbf{W}_{\infty;+}=\left\{\mathbf{x}=(x_i)_{i\in \mathbb{N}}\in \mathbb{R}^{\mathbb{N}}:x_1\ge x_2\ge x_3 \ge \cdots \ge 0\right\},\\
\mathbf{W}_{\infty;+}^\circ=\left\{\mathbf{x}=(x_i)_{i\in \mathbb{N}}\in \mathbb{R}^{\mathbb{N}}:x_1> x_2> x_3 > \cdots >0\right\}.
\end{align*}
The following infinite-dimensional spaces, in some sense extended versions of the $\mathbf{W}_{\infty;+}$ but with some local-compactness, will play an important role. Write $\mathbb{R}_+=[0,\infty)$.
\begin{defn} We define the following spaces:
\begin{align} 
\Upsilon&=\left\{\boldsymbol{\upsilon}=(\mathbf{x}^+,\mathbf{x}^-,\gamma,\delta)\in \mathbf{W}_{\infty;+}\times \mathbf{W}_{\infty;+}\times \mathbb{R}\times \mathbb{R}_+:\sum_{i=1}^\infty (x_i^+)^2+(x_i^-)^2\le \delta\right\},\label{Ypsilon}\\
\Upsilon_+&=\left\{\boldsymbol{v}=(\mathbf{x},\gamma)\in \mathbf{W}_{\infty;+}\times \mathbb{R}_+: \sum_{i=1}^\infty x_i\le \gamma \right\},\label{Ypsilon+}
\end{align}
endowed with product topology as subsets of $\mathbb{R}^{\mathbb{N}}\times \mathbb{R}^{\mathbb{N}}\times \mathbb{R}\times \mathbb{R}_+$ and $\mathbb{R}^{\mathbb{N}}\times \mathbb{R}_+$ respectively.
\end{defn}
Note that, $\Upsilon, \Upsilon_+$ are locally compact Polish spaces. These spaces are in fact intimately related to a distinguished class of entire functions called the Laguerre-Polya class. These entire functions are exactly the ones that arise as uniform limits in $\mathbb{C}$ of polynomials with only real (or only non-negative real) zeroes. Given $\boldsymbol{v} \in \Upsilon_+$  and $\boldsymbol{u}\in \Upsilon$  define the entire functions $\mathfrak{E}_{\boldsymbol{v}}^+$ and $\mathfrak{E}_{\boldsymbol{u}}$ by their Hadamard factorisations as follows (all these functions are normalised so that they are $1$ when evaluated at $z=0$):
\begin{align*}
\mathfrak{E}_{\boldsymbol{v}}^+(z)&=\mathrm{e}^{-(\gamma-\sum_{i=1}^\infty x_i)z}\prod_{i=1}^\infty \left(1-zx_i\right), \\
\mathfrak{E}_{\boldsymbol{u}}(z) &=\mathrm{e}^{-\gamma z-(\delta-\sum_{i=1}^\infty (x_i^+)^2+(x_i^-)^2)z^2}\prod_{i=1}^\infty(1-zx_i^+)\mathrm{e}^{zx_i^+}\prod_{i=1}^\infty(1+zx_i^-)\mathrm{e}^{-zx_i^-}.
\end{align*}
These functions are well-defined by virtue of the inequalities in the definitions of $\Upsilon$ and $\Upsilon_+$ in \eqref{Ypsilon} and \eqref{Ypsilon+} respectively. For a domain $\mathscr{D}\subseteq \mathbb{C}$  of the complex plane write $\mathsf{H}(\mathscr{D})$ for the space of holomorphic functions on $\mathscr{D}$ endowed with the topology of uniform on compact sets convergence which makes it into a Frechet space.

\begin{defn} We define the following subsets of $\mathsf{H}(\mathbb{C})$,
\begin{equation*}
\mathfrak{LP}_+=\{\mathfrak{E}^+_{\boldsymbol{v}};\boldsymbol{v}\in \Upsilon_+\}, \ \ \mathfrak{LP}=\{\mathfrak{E}_{\boldsymbol{u}};\boldsymbol{u}\in \Upsilon\},
\end{equation*}
endowed with the topology of uniform convergence on compact sets.
\end{defn}

We now turn to multiplicative Brownian motion on $\mathsf{GL}_N(\mathbb{C})$ \cite{NorrisRogersWilliams,Biane,JonesOConnell,BianeMatrixValuedBM, IpsenSchomerus,KempGLN,BCKP,AhnVanPeski}. Let us denote by $\mathbb{M}_N(\mathbb{C})$ the space of all $N\times N$ complex matrices and let $\mathsf{GL}_N(\mathbb{C})$
be the group of invertible ones.

\begin{defn}
For any $N\in\mathbb{N}$ and 
$\theta\in\mathbb{R}$, we define the drifted  multiplicative Brownian motion 
$(\mathsf{Y}_N^\theta(t))_{t\ge0}$ on $\mathsf{GL}_N(\mathbb{C})$, with drift $\theta$, as the unique strong solution to the matrix-valued SDE:
\begin{align}\label{thetaBMonGL}
\mathrm{d} \mathsf{Y}_N^\theta(t) = \mathsf{Y}_N^\theta(t)\, \mathrm{d} \boldsymbol{B}_N(t) + \theta \mathsf{Y}_N^\theta(t) \, \mathrm{d} t,
\end{align}
where
$\boldsymbol{B}_N$ is a matrix Brownian motion on $\mathbb{M}_N(\mathbb{C})$, that is,
\begin{align*}
\boldsymbol{B}_N(t) &=  \boldsymbol{B}_N^{(1)}(t) + \mathrm{i} \boldsymbol{B}_N^{(2)}(t),
\end{align*}
with $\boldsymbol{B}_N^{(1)},\boldsymbol{B}_N^{(2)}$ independent matrices with independent standard real Brownian motions as entries. 
   \end{defn}

The fact that the matrix SDE \eqref{thetaBMonGL} has a unique strong solution follows from standard results for SDEs with Lipschitz coefficients \cite{IkedaWatanabe,RevuzYor}. Write $\mathbf{A}^\dag$ for the conjugate transpose of a matrix $\mathbf{A}$ and fix $\theta\in \mathbb{R}$ that we drop from the notation. Let $(\mathfrak{x}_i^{(N)})_{i=1}^N$ be the squared singular values of $\mathsf{Y}_N$, arranged in non-increasing order, or equivalently the eigenvalues of $\mathsf{R}_N=\mathsf{Y}_N^\dag \mathsf{Y}_N$ which solves the matrix SDE, with $\mathbf{I}_N$ the identity matrix,
\begin{equation*}
\mathrm{d}\mathsf{R}_N(t)=\mathsf{R}_N(t)\, \mathrm{d} \boldsymbol{B}_N(t) + \mathrm{d} \boldsymbol{B}_N(t)^\dag \, \mathsf{R}_N(t) + 2\theta\mathsf{R}_N(t)\, \mathrm{d}t + 2\operatorname{Tr}(\mathsf{R}_N(t))\, \mathbf{I}_N\, \mathrm{d}t.
\end{equation*}
Standard results  \cite{MatrixYW,Graczyk-Malecki} give that $(\mathfrak{x}_i^{(N)})_{i=1}^N$ is the unique solution of the equation in $\mathbf{W}_{N;+}$ \begin{align}\label{evalX-SDE}
\mathrm{d} \mathfrak{x}_i^{(N)}(t) = 2\mathfrak{x}_i^{(N)}(t) \, \mathrm{d} \mathsf{w}_i(t) + 
2\mathfrak{x}_i^{(N)}(t)\left( (1+\theta)
+
\sum_{j=1, j \neq i}^N \frac{\mathfrak{x}_i^{(N)}(t) + \mathfrak{x}_j^{(N)}(t)}{\mathfrak{x}_i^{(N)}(t) - \mathfrak{x}_j^{(N)}(t)}\right)\mathrm{d}t,\  \  i\in \llbracket 1, N \rrbracket,
\end{align}
with the $\mathsf{w}_i$ independent standard Brownian motions and $\llbracket k,\ell \rrbracket =  \{k, k+1, \ldots, \ell\}$. The process that is most commonly considered in the literature is the logarithms of the (squared) singular values of $\mathsf{Y}_N$, in particular one lets
\begin{align}\label{XiDef}
  \xi_i^{(N)}(t)\overset{\textnormal{def}}{=}
  \displaystyle\log \mathfrak{x}_i^{(N)}(t/4), \  \ t\ge 0,\  \ i\in \llbracket 1, N \rrbracket,
\end{align}
which by an application of It\^o's formula solves
\begin{align*}
\mathrm{d}\xi_i^{(N)}(t)&= \mathrm{d} \mathsf{w}_i(t)
+ 
\frac{\theta}{2}\mathrm{d}t
+
 \frac{1}{2}\sum_{\substack{j=1, j \neq i}}^N \coth\left( \frac{\xi_i^{(N)}(t) - \xi_j^{(N)}(t)}{2} \right) \mathrm{d}t,
\  \ i \in \llbracket 1, N \rrbracket.
\end{align*}
Interestingly, the choice $\theta=0$ of this equation arises as a particular case of radial Heckman-Opdam processes
studied in \cite{Heckman-Opdam}. Finally, we define the following rescaled process, which is the one that we will be working with,
\begin{align}\label{x-rsc}
    \mathsf{x}_i^{(N)}(t)
    \overset{\textnormal{def}}{=}
    \mathrm{e}^{-Nt/2}\mathfrak{x}_i^{(N)}(t/4), \  \ t\ge 0,\  \ i \in \llbracket 1, N \rrbracket,
\end{align}
which by It\^o's formula solves
\begin{equation}\label{GLinitial}
\mathrm{d}\mathsf{x}_i^{(N)}(t) = \mathsf{x}_i^{(N)}(t) \mathrm{d}\mathsf{w}_i(t) +\frac{\theta}{2}\mathsf{x}_i^{(N)}(t)\mathrm{d}t+\sum_{j=1,j\neq i}^N\frac{\mathsf{x}_i^{(N)}(t)\mathsf{x}^{(N)}_j(t)}{\mathsf{x}^{(N)}_i(t)-\mathsf{x}^{(N)}_j(t)}\mathrm{d}t, \ \  i\in \llbracket 1, N\rrbracket.   
\end{equation}
As for \eqref{evalX-SDE}, from \cite{AGZ,Graczyk-Malecki} this SDE has a unique strong solution starting from anywhere in $\mathbf{W}_{N;+}$. Finally, define the rescaled reverse characteristic polynomial associated to \eqref{GLinitial} by 
\begin{equation*}
\mathfrak{gl}_{t;N}(z)=\prod_{i=1}^N\left(1-\frac{\mathsf{x}_i^{(N)}(t)z}{N}\right).
\end{equation*}
Of course, $\mathfrak{gl}_{t;N}$ depends on $\theta$ but we suppress it from the notation. Recall from \cite{Kallenberg} the definition of a $\mathrm{C}_0$-Feller semigroup and the associated Feller process. A diffusion process is a strong Markov process with continuous sample paths. The precise definition of Gibbs resampling property \cite{CorwinHammond1} will be given in Section \ref{SubsectionGibbs}. We write $\mathrm{C}(\mathbb{R}_{+};\mathscr{X})$ for the space of continuous functions from $\mathbb{R}_+$ to a metric space $\mathscr{X}$ endowed with the topology of uniform convergence on compact sets in time. Finally, we define $\mathsf{x}_i^{(N)}$ for $i>N$ to be identically zero. We can now state our first main result. 

\begin{thm}\label{ThmMainGLN}
Let $\theta\in \mathbb{R}$ be fixed. Then, there exists a unique Feller diffusion process $(\mathbf{X}_{GL}^t)_{t\ge 0}$ on $\Upsilon_+$ such that whenever, 
\begin{equation*}
\left((N^{-1}\mathsf{x}_i^{(N)}(0))_{i\in \mathbb{N}},N^{-1}\sum_{i=1}^\infty \mathsf{x}_i^{(N)}(0)\right)  \overset{N \to \infty}{\longrightarrow} \boldsymbol{v}=((x_i)_{i\in \mathbb{N}},\gamma)\in \Upsilon_+,
\end{equation*}
we have, as $N \to \infty$,
\begin{equation*}
\left((N^{-1}\mathsf{x}_i^{(N)}(\bullet))_{i\in \mathbb{N}},N^{-1}\sum_{i=1}^\infty \mathsf{x}_i^{(N)}(\bullet)\right) \overset{\mathrm{d}}{\longrightarrow} \left((\mathsf{x}_i(\bullet))_{i\in \mathbb{N}},\boldsymbol{\gamma}(\bullet)\right)\overset{\mathrm{def}}{=} \mathbf{X}_{GL}^\bullet
\end{equation*}
in $\mathrm{C}(\mathbb{R}_+;\Upsilon_+)$
where $\mathbf{X}_{GL}^0=\boldsymbol{v}$, and $\overset{\mathrm{d}}{\longrightarrow}$ denotes convergence in distribution.
Furthermore, for any $\gamma\ge \sum_{i=1}^\infty x_i$, whenever $(x_i)_{i\in \mathbb{N}}\in \mathbf{W}_{\infty;+}^\circ$, the projection of $\mathbf{X}_{GL}^\bullet$ on the $(\mathsf{x}_i(\bullet))_{i\in \mathbb{N}}$ solves the infinite system of SDE
\begin{equation}\label{ISDEGL}
\mathsf{x}_i(t)= x_i+\int_0^t\mathsf{x}_i(s)\mathrm{d}\mathsf{w}_i(s) + \frac{\theta}{2}\int_0^t\mathsf{x}_i(s)\mathrm{d}s
 +\int_0^t\sum_{j=1,j\neq i}^{\infty}\frac{\mathsf{x}_i(s)\mathsf{x}_j(s)}{\mathsf{x}_i(s)-\mathsf{x}_j(s)}\mathrm{d}s, \  \forall t\ge 0 , \   i \in \mathbb{N},
 \end{equation}
with independent standard Brownian motions $(\mathsf{w}_i)_{i\in \mathbb{N}}$ and $(\mathsf{x}_i(t))_{i\in \mathbb{N},t\in \mathbb{R}_+}$ viewed as a line ensemble satisfies a  Gibbs  property with respect to exponential Brownian bridges. Moreover, as $N\to \infty$,
\begin{equation*}
\mathfrak{gl}_{\bullet;N}\overset{\mathrm{d}}{\longrightarrow} \mathfrak{gl}_\bullet,
\end{equation*}
in $\mathrm{C}(\mathbb{R}_+;\mathfrak{LP}_+)$ where $\mathfrak{gl}_t(z)=\mathfrak{E}_{\mathbf{X}_{GL}^t}^+(z)$ is a Markov process with values in $\mathfrak{LP}_+$ which solves the following stochastic partial differential equation 
\begin{equation}\label{GLequation}
\mathrm{d}\mathfrak{gl}_t(z)=\mathrm{d}\mathsf{M}_t(z;\mathfrak{gl})+\left[\frac{\theta}{2}z\partial_z\mathfrak{gl}_t(z)
    -
    \frac{z^2}{2}\partial_{z}^2\mathfrak{gl}_t(z)\right]\mathrm{d}t
\end{equation}
where $\mathsf{M}\in \mathrm{C}(\mathbb{R}_+;\mathsf{H}(\mathbb{C}))$ and for all $z\in \mathbb{C}$, $t\mapsto \mathsf{M}_t(z;\mathfrak{gl})$ are continuous local martingales with covariation, for $z,w \in \mathbb{C}$,
\begin{equation}\label{CovarEq}
\langle \mathsf{M}_t(z;\mathfrak{gl}),\mathsf{M}_t(w;\mathfrak{gl})\rangle=\frac{zw}{z-w}\int_0^t\left(\mathfrak{gl}_s(z)\partial_w\mathfrak{gl}_s(w)-\mathfrak{gl}_s(w)\partial_z\mathfrak{gl}_s(z)\right)\mathrm{d}s.
\end{equation}
Finally, for $\theta<0$, we have,  as $t \longrightarrow \infty$,  $\mathfrak{gl}_t(z)\overset{\mathrm{d}}{\longrightarrow} 1$ in $\mathfrak{LP}_+$.
\end{thm}

A few comments are in order.

First, observe that the right hand side of \eqref{CovarEq} extends in an analytic way to $z=w$ to obtain for the quadratic variation of $t\mapsto \mathsf{M}_t(z;\mathfrak{gl})$,
\begin{equation}
\langle \mathsf{M}_t(z;\mathfrak{gl})\rangle=z^2\int_0^t\left((\partial_z\mathfrak{gl}_s(z))^2-\mathfrak{gl}_s(z)\partial^2_z\mathfrak{gl}_s(z)\right)\mathrm{d}s.
\end{equation}
%The fact that this is actually positive for $z\in \mathbb{R}$ is not completely obvious. It is a consequence of the fact that $\mathfrak{gl}_t\in \mathfrak{LP}_+$.
To our knowledge, the stochastic equation \eqref{GLequation}, with the martingale term depending in this particular non-linear way on $\mathfrak{gl}_{t}$, has not appeared before in the stochastic pde literature. We also derive an equation for the transform $\psi_t^{GL}(z)=\partial_z\log\mathfrak{gl}_t(z)$ below for which the martingale term depends linearly on $\psi$ but the drift term involves Burgers-type non-linearities. The form of the covariation \eqref{CovarEq} is reminiscent, had $\mathfrak{gl}_t$ been deterministic, to the form of correlation kernels and integrable operators from random matrix theory and Riemann-Hilbert problems \cite{IntegrableOperators,Deift,AGZ,ForresterLogGas}. It would be interesting to investigate existence and uniqueness properties of the equation \eqref{GLequation}, and the analogous ones \eqref{RVspde}, \eqref{HPspde} below, directly at the level of these equations instead of using an approximation procedure from finite-dimensional dynamics as we do here.

The scaled process $(N^{-1}\mathsf{x}_i^{(N)}(\bullet))_{i\in \mathbb{N}}$, for $\theta=0$, starting from $\mathsf{x}^{(N)}(0)=(1,\dots,1)$ exactly corresponds in logarithmic coordinates to the scaled dynamics considered by Ahn in \cite{Ahn} (Ahn's $\boldsymbol{\xi}_i^{(N)}(t/4)$ is our $\xi_i^{(N)}(t)$ to be precise):
\begin{equation}\label{AhnRescaledProcess}
\left(\xi_i^{(N)}\left(t\right)-\frac{Nt}{2}-\log N\right)_{i\in \mathbb{N},t\in \mathbb{R}_+},
\end{equation}
and whose $N \to \infty $ limit also arises in products of random matrices \cite{Ahn,LiuWangWang,Akemann2019integrable,AkemannUniversality}. 

Finally, regarding equation \eqref{ISDEGL} on the dynamics of the paths, it can be shown that the solutions
constructed in Theorem \ref{ThmMainGLN} have different laws for different initial parameters $\gamma$ but there is only one out of them for which almost surely $t\mapsto \sum_{i=1}^\infty \mathsf{x}_i(t)$ is continuous on all of $\mathbb{R}_+$ given by  the choice $\gamma=\sum_{i=1}^\infty x_i$. The argument is completely analogous to the one from Section 4 of \cite{AM24} and we do not reproduce it here. We will say more about uniqueness after Theorem \ref{ThmMainStat}.

We now move on to consider two models from \cite{AM24,HuaPickrellDiffusions} with two distinguished random entire functions \cite{ChhaibiNajnudelNikeghbali,ValkoVirag2,LiValko,Theo-RandomEntireFunctions} as stationary measures. To define these functions we need some preliminaries but first let us define the dynamics themselves. 

Define the Rider-Valk\'{o} model \cite{RiderValko,GrabschTexier,MatrixKesten} as the unique strong solution of the SDE starting anywhere in $\mathbf{W}_{N;+}$, with $\nu \in \mathbb{R}_+$,
\begin{equation}\label{RVmodel}
 \mathrm{d}\mathsf{z}_i^{(N)}(t) = \mathsf{z}_i^{(N)}(t) \mathrm{d}\mathsf{w}_i(t) -\frac{\nu}{2}\mathsf{z}_i^{(N)}(t)\mathrm{d}t+\frac{1}{2}\mathrm{d}t+\sum_{j=1,j\neq i}^N\frac{\mathsf{z}_i^{(N)}(t)\mathsf{z}^{(N)}_j(t)}{\mathsf{z}^{(N)}_i(t)-\mathsf{z}^{(N)}_j(t)}\mathrm{d}t, \ \  i\in \llbracket 1, N \rrbracket.   
\end{equation}
and almost surely for all $t>0$ living in the interior $\mathbf{W}_{N;+}^\circ$ and then the dynamical Cauchy model or Hua-Pickrell diffusions \cite{HuaPickrellDiffusions, AM24,AristaDemni,AuerVoit} as the unique strong solution to the SDE,  with $\mathfrak{s}\in \mathbb{C}$
\begin{align}\label{Hua-PickrellIntro}
\mathrm{d}\mathsf{y}_i^{(N)}(t)&=\sqrt{2\left((\mathsf{y}_i^{(N)}(t))^2+1\right)}\mathrm{d}\mathsf{w}_i(t)+2\left[-\Re(\mathfrak{s})\mathsf{y}_i^{(N)}(t)+\Im(\mathfrak{s})\right]\mathrm{d}t\nonumber\\
&+2\sum_{j=1,j\neq i}^N \frac{\mathsf{y}_i^{(N)}(t)\mathsf{y}_j^{(N)}(t)+1}{\mathsf{y}_i^{(N)}(t)-\mathsf{y}_j^{(N)}(t)}\mathrm{d}t, \ \ i\in \llbracket 1, N \rrbracket,
\end{align}
starting anywhere in the ordered chamber $\mathbf{W}_N$ over the whole of $\mathbb{R}$:
\begin{equation*}
\mathbf{W}_N=\left\{\mathbf{x}=(x_i)_{i=1}^N\in \mathbb{R}^N:x_1\ge x_2 \ge \cdots \ge x_N\right\},
\end{equation*}
and almost surely for all $t>0$ living in the interior $\mathbf{W}_N^\circ$. Various properties of these models have been studied in \cite{AM24,HuaPickrellDiffusions, AristaDemni,AuerVoit,RiderValko}. There are also associated matrix valued diffusions to \eqref{RVmodel} and \eqref{Hua-PickrellIntro}, see \cite{HuaPickrellDiffusions,AM24,MatrixBougerol}. Note that, although we use the same notation $\mathsf{w}_i$ for independent standard Brownian motions in various different SDEs our statements in this paper will make no claims on explicit couplings driven with the same Brownian motions. 

Now, we define the Hua-Pickrell \cite{BorodinOlshanskiErgodic,ForresterWitteCauchy} point process $\mathcal{HP}_{\mathfrak{s}}$ on $\mathbb{R}\backslash\{0\}$, for $\mathfrak{s}>-1/2$ as the determinantal point process \cite{JohanssonDeterminantal,BorodinDeterminantal,ForresterLogGas}  with kernel $\mathcal{K}_\mathfrak{s}(x,y)$ given by
\begin{equation*}
\mathcal{K}_\mathfrak{s}(x,y)=\frac{2^{4\mathfrak{s}-1}}{\pi}\frac{\Gamma(\mathfrak{s}+1)^2\Gamma\left(\mathfrak{s}+\frac{1}{2}\right)\Gamma\left(\mathfrak{s}+\frac{3}{2}\right)}{\Gamma(2\mathfrak{s}+1)\Gamma(2\mathfrak{s}+2)}\frac{J_{\mathfrak{s}-1/2}\left(\frac{1}{|x|}\right)J_{\mathfrak{s}+1/2}\left(\frac{1}{|y|}\right)-J_{\mathfrak{s}-1/2}\left(\frac{1}{|y|}\right)J_{\mathfrak{s}+1/2}\left(\frac{1}{|x|}\right)}{(x-y)|x|^{\frac{1}{2}}|y|^\frac{1}{2}},
\end{equation*}
where $J_\alpha$ is the Bessel function of the first kind with index $\alpha$. For $\mathfrak{s}=0$ and under the $x\mapsto 1/x$ transformation this specialises to the famous sine point process \cite{ForresterLogGas}. The process $\mathcal{HP}_\mathfrak{s}$ can also be defined for $\mathfrak{s}\in \mathbb{C}$ with $\Re(\mathfrak{s})>-1/2$ but we will not consider this more general case here. Define the inverse Bessel point process $\mathcal{IB}_\nu$ on $(0,\infty)$, with parameter $\nu>-1$ as the determinantal point process with correlation kernel $\mathcal{R}_\nu(x,y)$
\begin{equation*}
\mathcal{R}_\nu(x,y)=\frac{8\sqrt{2}}{xy}\frac{J_{\nu+1}\left(\frac{2\sqrt{2}}{x^{1/2}}\right)J_{\nu}\left(\frac{2\sqrt{2}}{y^{1/2}}\right)-J_{\nu+1}\left(\frac{2\sqrt{2}}{y^{1/2}}\right)J_{\nu}\left(\frac{2\sqrt{2}}{x^{1/2}}\right)}{x-y}.
\end{equation*}
Under the transformation $x\mapsto 1/x$ one obtains the more standard Bessel point process \cite{ForresterBessel,ForresterLogGas} (and thus the name inverse Bessel). 

We now come to the definition of the random entire functions of interest (the infinite products defining them converge uniformly in $z\in \mathbb{C}$ almost surely, see for example \cite{Theo-RandomEntireFunctions}).

\begin{defn}
For $\mathfrak{s}>-1/2$ and $\nu>-1$ define the Hua-Pickrell $\mathbb{HP}_\mathfrak{s}$ and Bessel $\mathbb{B}_\nu$ stochastic zeta functions by:
\begin{align*}
    \mathbb{HP}_\mathfrak{s}(z)=\lim_{k\to \infty}\prod_{\substack{\mathrm{x}\in \mathcal{HP}_\mathfrak{s}\\|\mathrm{x}|>k^{-2}}}\left(1-z\mathrm{x}\right), \ \  \
    \mathbb{B}_{\nu}(z)=\prod_{\mathrm{x}\in \mathcal{IB}_{\nu}}(1-z\mathrm{x}), \ \ \forall z \in \mathbb{C}.
\end{align*}
\end{defn}
These random entire functions are known as stochastic zeta functions \cite{ChhaibiNajnudelNikeghbali,ValkoVirag2} and are closely related to stochastic operators associated to random matrices \cite{RamirezRiderVirag,RamirezRider,ValkoVirag1}. They have been intensively studied in recent years \cite{ChhaibiNajnudelNikeghbali,ValkoVirag2,LiValko,LambertPaquette1,LambertPaquette2,NajnudelNikeghbali,Theo-RandomEntireFunctions}. The moments of their Taylor coefficients and certain integral statistics of them appear as leading order coefficients in asymptotics of various generalised moments of characteristic polynomials of random unitary matrices, see \cite{AGKW1,AGKW2,AssiotisNajnudel}. These leading order coefficients give conjectures in number theory for averages of the Riemann $\zeta$-function and families of L-functions following the Keating-Snaith philosophy \cite{KeatingSnaith1,KeatingSnaith2,BaileyKeating}.

Finally, define the rescaled reverse characteristic polynomials associated to the dynamics \eqref{RVmodel} and \eqref{Hua-PickrellIntro},
\begin{equation*}
\mathfrak{rv}_{t;N}(z)=\prod_{i=1}^N\left(1-\frac{\mathsf{z}_i^{(N)}(t)z}{N}\right), \ \ \mathfrak{hp}_{t;N}(z)=\prod_{i=1}^N\left(1-\frac{\mathsf{y}_i^{(N)}(t)z}{N}\right).
\end{equation*}
Clearly, these depend on $\nu$ and $\mathfrak{s}$ respectively but we suppress it from the notation. We then have the following theorem which is obtained by combining the new perspective of looking at the evolution of the functions and our results from \cite{AM24}.

\begin{thm}\label{ThmMainStat}
Let $\nu\in \mathbb{R}$ and $\mathfrak{s} \in \mathbb{C}$ be fixed. If $\mathfrak{rv}_{0;N} \overset{N \to \infty}{\longrightarrow} \mathfrak{E}_{\boldsymbol{v}}^+$ in $\mathfrak{LP}_+$, then as $N\to \infty$,
\begin{equation*}
\mathfrak{rv}_{\bullet;N}\overset{\mathrm{d}}{\longrightarrow} \mathfrak{rv}_\bullet,
\end{equation*}
in $\mathrm{C}(\mathbb{R}_+;\mathfrak{LP}_+)$ and  $(\mathfrak{rv}_t)_{t\ge 0}$ is a Markov process with values in $\mathfrak{LP}_+$ which solves the following stochastic partial differential equation 
\begin{equation}\label{RVspde}
\mathrm{d}\mathfrak{rv}_t(z)=\mathrm{d}\mathsf{M}_t(z;\mathfrak{rv})+\left[-\frac{z}{2}\mathfrak{rv}_t(z)-\frac{\nu}{2}z\partial_z\mathfrak{rv}_t(z)
    -
    \frac{z^2}{2}\partial_{z}^2\mathfrak{rv}_t(z)\right]\mathrm{d}t
\end{equation}
where $\mathsf{M}\in \mathrm{C}(\mathbb{R}_+;\mathsf{H}(\mathbb{C}))$ and for all $z\in \mathbb{C}$, $t\mapsto \mathsf{M}_t(z;\mathfrak{rv})$ are  continuous local martingales with covariation, for $z,w \in \mathbb{C}$,
\begin{equation}
\langle \mathsf{M}_t(z;\mathfrak{rv}),\mathsf{M}_t(w;\mathfrak{rv})\rangle=\frac{zw}{z-w}\int_0^t\left(\mathfrak{rv}_s(z)\partial_w\mathfrak{rv}_s(w)-\mathfrak{rv}_s(w)\partial_z\mathfrak{rv}_s(z)\right)\mathrm{d}s.
\end{equation}
Moreover, if $\mathfrak{hp}_{0;N} \overset{N \to \infty}{\longrightarrow} \mathfrak{E}_{\boldsymbol{v}}$ in $\mathfrak{LP}$, then as $N\to \infty$,
\begin{equation*}
\mathfrak{hp}_{\bullet;N}\overset{\mathrm{d}}{\longrightarrow} \mathfrak{hp}_\bullet,
\end{equation*}
in $\mathrm{C}(\mathbb{R}_+;\mathfrak{LP})$ and  $(\mathfrak{hp}_t)_{t\ge 0}$ is a Markov process with values in $\mathfrak{LP}$ which solves the following stochastic partial differential equation 
\begin{equation}\label{HPspde}
\mathrm{d}\mathfrak{hp}_t(z)=\mathrm{d}\mathsf{M}_t(z;\mathfrak{hp})+\left[-(2\Im(\mathfrak{s})z+z^2)\mathfrak{hp}_t(z)-2\Re(\mathfrak{s})z\partial_z\mathfrak{hp}_t(z)-z^2\partial_z^2\mathfrak{hp}_t(z)\right]\mathrm{d}t
\end{equation}
where $\mathsf{M}\in \mathrm{C}(\mathbb{R}_+;\mathsf{H}(\mathbb{C}))$ and for all $z\in \mathbb{C}$, $t\mapsto \mathsf{M}_t(z;\mathfrak{hp}_t)$ are continuous local martingales with covariation, for $z,w \in \mathbb{C}$,
\begin{equation}
\langle \mathsf{M}_t(z;\mathfrak{hp}),\mathsf{M}_t(w;\mathfrak{hp})\rangle=2\frac{zw}{z-w}\int_0^t\left(\mathfrak{hp}_s(z)\partial_w\mathfrak{hp}_s(w)-\mathfrak{hp}_s(w)\partial_z\mathfrak{hp}_s(z)\right)\mathrm{d}s.
\end{equation}
Finally, for $\nu>-1$ and any $\mathfrak{rv}_0\in \mathfrak{LP}_+$ we have, as $t \longrightarrow \infty$,
\begin{equation}\label{ConvToEqIB}
\mathfrak{rv}_t \overset{\mathrm{d}}{\longrightarrow} \mathbb{B}_\nu , \ \ \textnormal{in } \mathfrak{LP}_+, 
\end{equation}
while for $\mathfrak{s}\in \mathbb{R}$, with $\mathfrak{s}>-1/2$, and any $\mathfrak{hp}_0\in \mathfrak{LP}$, as $t \to \infty$,
\begin{equation}\label{ConvToEqHP}
\mathfrak{hp}_t \overset{\mathrm{d}}{\longrightarrow} \mathbb{HP}_\mathfrak{s} , \ \ \textnormal{in } \mathfrak{LP}.
\end{equation}
\end{thm}

For $\mathfrak{s}=0$, $(\mathfrak{hp}_t)_{t\ge 0}$ gives the first natural dynamical model of the stochastic zeta function of the sine point process \cite{ChhaibiNajnudelNikeghbali,ValkoVirag2}. A more well-known dynamic related to the sine point process is the bulk limit of Dyson Brownian motion \cite{KatoriTanemura, Tsai,Osada1,Osada2,Osada3,KawamotoOsada, OsadaTanemura, SuzukiErgodicity,SuzukiCurvature}. However, it is unclear how one could associate to it some natural evolution on entire functions. In a different direction, for macroscopic/mesoscopic (unlike the microscopic regime we are concerned with) limits of various random matrix dynamics that have been intensively studied in the literature recently in relation to log-correlated fields and Gaussian multiplicative chaos, see \cite{BourgadeFalconet,BourgadeCipolloniHuang,Keles}.

Theorem \ref{ThmMainStat} gives a new dynamical way to study properties of the entire functions $\mathbb{B}_\nu$ and $\mathbb{HP}_\mathfrak{s}$ e.g. formally the moments are the $t \to \infty $ limits of solutions to systems of differential equations arising by taking expectations in \eqref{RVspde} and \eqref{HPspde}. 

It is interesting to note that, as we have shown in \cite{AM24}, the limiting rescaled dynamics of Rider-Valk\'{o} \eqref{RVmodel} or equivalently the reciprocals of zeros of $(\mathfrak{rv}_t)_{t\ge0} $ also solve \eqref{ISDEGL} with the identification $\nu=-\theta$. In particular, the equation \eqref{ISDEGL} on its own cannot distinguish between dynamics which are qualitatively distinct. On the other hand, the equations \eqref{GLequation} and \eqref{RVspde} for $(\mathfrak{gl}_t)_{t\ge 0}$ and $(\mathfrak{rv}_t)_{t\ge 0}$, which in principle contain more information, are actually different and would be interesting if they can characterise their solutions. A different way to remedy the non-uniqueness of \eqref{ISDEGL} is to look for uniqueness within ``rigid-path" solutions to it i.e. solutions which follow deterministic paths $(\boldsymbol{\alpha}_i(\bullet))_{i\in \mathbb{N}}$ with lower order random fluctuations. In particular the limiting Rider-Valk\'{o} dynamics and $\mathsf{GL}_N(\mathbb{C})$ dynamics, starting from the same initial condition, would then follow very different deterministic paths, as can also be seen from simulations. We have some arguments in the direction of ``rigid-path" uniqueness but currently no robust method to show that from generic initial conditions any of the solutions we constructed are indeed ``rigid-path" solutions (we believe they are) and to find the $(\boldsymbol{\alpha}_i(\bullet))_{i\in \mathbb{N}}$ explicitly. It is not hard to show that rigidity of the paths boils down to fixed-time rigidity and a sufficiently strong modulus of continuity control. Understanding this conjectural non-equilibrium path-rigidity is an outstanding open problem.

Finally, it is worth mentioning that convergence to equilibrium for $(\mathfrak{hp}_t)_{t\ge 0}$ also holds for $\mathfrak{s}\in \mathbb{C}$, with $\Re(\mathfrak{s})>-1/2$, however the description of the limiting entire function is not as explicitly known as for $\mathbb{HP}_\mathfrak{s}$ with real $\mathfrak{s}$.

The dynamical Cauchy model $(\mathfrak{hp}_t)_{t\ge0}$ is significantly more complicated than $(\mathfrak{gl}_t)_{t\ge0}$ and $(\mathfrak{rv}_t)_{t\ge0}$ and we do not have as complete a description for the dynamics of its zeros. Our final task in this paper is to prove some rather non-trivial conditional results and offer a precise conjecture. We need to introduce the following assumption. 
\begin{ass}\label{Assumption}
We assume that the ordered zeroes $(\mathfrak{y}_i(t))_{i\in \mathbb{Z}\backslash \{0\}}$ of $\mathfrak{hp}_t$ satisfy almost surely for all $t\ge 0$,
\begin{equation*}
\cdots <\mathfrak{y}_{-3}(t)<\mathfrak{y}_{-2}(t)<\mathfrak{y}_{-1}(t)<0< \mathfrak{y}_1(t)<\mathfrak{y}_2(t)<\mathfrak{y}_3(t)<\cdots, 
\end{equation*}
with no explosions to $\pm \infty$ in finite time.
\end{ass}

We believe, but cannot prove, that Assumption \ref{Assumption} is true; in fact we believe it is true in substantial generality. A sufficient condition may simply be its validity for the initial condition at $t=0$. 

Let us write $(\mathsf{y}_i(t))_{i\in \mathbb{Z}\backslash \{0\},t\in \mathbb{R}_+}$ for the reciprocals of the zeros of $(\mathfrak{hp}_t)_{t\ge 0}$:
\begin{equation}\label{RecipZeros}
\mathsf{y}_i(t)\overset{\mathrm{def}}{=}\frac{1}{\mathfrak{y}_i(t)}, \ \ \forall t\ge0, \ i\in \mathbb{Z}\backslash \{0\}.
\end{equation}
We then have the following equations describing their dynamics.

\begin{prop}\label{HPISDE} Let $\mathfrak{s}\in \mathbb{C}$.
Under Assumption \ref{Assumption} for the zeroes of $(\mathfrak{hp}_t)_{t\ge 0}$, we have
\begin{align*}
&\mathrm{d}\mathsf{y}_i(t)
  = 
 \sqrt{2}\mathsf{y}_i(t)
 \mathrm{d}\mathsf{w}_i(t)
-2\left(\Re{(\mathfrak{s})}
+1\right)\mathsf{y}_i(t)\mathrm{d}t
  +
 2
 \sum_{j\in\mathbb{Z}\setminus\{0,i\}}
\frac{\mathsf{y}_j^2(t)}{\mathsf{y}_i(t)-\mathsf{y}_j(t)}
\mathrm{d}t
+2
\boldsymbol{\gamma}_{HP}(t)
\mathrm{d}t,  i\in\mathbb{Z}\setminus\{0\},\\
&\mathrm{d}\boldsymbol{\gamma}_{HP}(t)
=
-\partial_z \mathrm{d}\mathsf{M}_t(z;\mathfrak{hp})\big|_{z=0}
+
2\left(\Im{(\mathfrak{s})}-\Re{(\mathfrak{s})}\boldsymbol{\gamma}_{HP}(t)\right)\mathrm{d}t.
\end{align*} 
\end{prop}

For $\mathfrak{s}\in \mathbb{R}$, $\mathfrak{s}>-1/2$, starting $(\mathfrak{hp}_t)_{t\ge 0}$ from the invariant measure, for any fixed time $t\ge 0$, it is known that $\boldsymbol{\gamma}_{HP}(t)$ has the same distribution as the principal value sum of the $\mathsf{y}_i$'s. If a dynamical extension of this result holds one then gets a closed equation for the $\mathsf{y}_i$'s. We thus naturally arrive at the following conjecture. 

\begin{conj}\label{ConjHP}
Let $\mathfrak{s} > -\tfrac{1}{2}$ and $\mathfrak{hp}_0\overset{\mathrm{d}}{=}\mathbb{HP}_\mathfrak{s}$. Then, the reciprocals $(\mathsf{y}_{i}(\bullet))_{i\in \mathbb{Z}\backslash\{0\}}$ of zeroes of $\mathfrak{hp}_{\bullet}$ satisfy the equation
\begin{equation*}
\mathrm{d}\mathsf{y}_i(t)
= \sqrt{2}\mathsf{y}_i(t)\mathrm{d}\mathsf{w}_i(t)
- 2\mathfrak{s}\mathsf{y}_i(t)\mathrm{d}t
+ 2 \lim_{\varepsilon\to 0}
\sum_{\substack{|\mathsf{y}_j(t)|>\varepsilon\\ j\in \mathbb{Z}\setminus\{0,i\}}}
\frac{\mathsf{y}_i(t)\mathsf{y}_j(t)}{\mathsf{y}_i(t)-\mathsf{y}_j(t)}\mathrm{d}t, \ \ \forall {i}\in \mathbb{Z}\backslash \{0\}.
\end{equation*}
\end{conj}

It is plausible that the conjecture is true also from general initial conditions which qualitatively look like the equilibrium measure. For initial conditions far from equilibrium or for non-real parameter $\mathfrak{s}\in \mathbb{C}$ it is less clear to us what happens.

\subsection{On the proof}

Let us finally comment on the proof. The starting point in our analysis is the observation that singular values of multiplicative Brownian motion on $\mathsf{GL}_N(\mathbb{C})$ satisfy a certain consistency relation as $N$ varies which allows us to make use of the method of intertwiners of Borodin and Olshanski \cite{BorodinOlshanskiMarkov,OlshanskiICM,OlshanskiLectureNotes}. This is a general algebraic formalism which constructs an abstract Feller-Markov process on the boundary of a projective system from a tower of consistent Feller processes on each layer, see \cite{BorodinOlshanskiMarkov,OlshanskiLectureNotes,AM24} for detailed expositions. The method was initially applied to discrete dynamics on partitions \cite{BorodinOlshanskiMarkov,BorodinOlshanskiThoma,BorodinGorin,OlshanskiApproximation,OlshanskiLectureNotes,OlshanskiICM} and more recently it was realised that it can be successfully applied to random matrix dynamics \cite{HuaPickrellDiffusions,AM24,BufetovKawamoto,BufetovKawamoto2}. Once this abstract process is constructed, the difficulty becomes describing it explicitly. In the random matrix context this was first achieved in \cite{AM24} for the Rider-Valk\'{o} model \eqref{RVmodel} and in this paper we go one step further for Brownian motion on $\mathsf{GL}_N(\mathbb{C})$ for which an additional property in Gibbs resampling can be established. The description of the dynamics of the limiting rescaled characteristic polynomials in terms of a closed spde is new for all three models we consider in this paper.

Ahn in his proof of convergence of \eqref{AhnRescaledProcess} with the special initial condition of the identity matrix for the $\mathsf{GL}_N(\mathbb{C})$ Brownian motion uses the Gibbs property (also present for general initial conditions) and explicit computations for the Laplace transform of the finite-dimensional distributions of \eqref{AhnRescaledProcess} and associated determinantal correlation kernel. The correlation kernel from the identity is quite special as one can relate it, see \cite{JonesOConnell,Katori,KatoriTakahashi}, to the correlation kernel of standard Dyson Brownian motion from equidistant initial condition \cite{JohanssonUniversality,JohanssonNumberVariance, KatoriTanemura} which is very amenable to asymptotic analysis \cite{Ahn,Mustazee}. In fact, it is an open problem to obtain the limit of the correlations from other initial conditions. Our proof, on the other hand requires minimal computations but granted it relies on the machinery built in \cite{AM24} which is used as a black box.

Regarding obtaining the singular log-interacting SDE \eqref{ISDEGL} for the dynamics of $(\mathsf{x}_i)_{i\in \mathbb{N}}$ we do that via the equation for $(\mathfrak{gl}_t)_{t\ge 0}$ or more precisely its logarithmic derivative. Schematically, we follow
\begin{equation}\label{Scheme}
\textnormal{Eq. } \eqref{GLequation} \textnormal{ for } \mathfrak{gl}_t \overset{\psi^{GL}_t(z)=\partial_z\log \mathfrak{gl}_t(z)}{\longrightarrow} \textnormal{Eq. } \eqref{psiEqu} \textnormal{ for } \mathfrak{\psi}_t^{GL}  \overset{\textnormal{Non-col.} \& \textnormal{Res.calc.}}{\longrightarrow} \textnormal{Eq. } \eqref{ISDEGL} \textnormal{ for } (\mathsf{x}_i)_{i\in \mathbb{N}}.
\end{equation}
A crucial non-intersection property of the paths is obtained from certain Lyapunov functions exactly as in \cite{AM24} or in principle the Gibbs resampling could be used, see Remarks \ref{RmkGibbsInitial} and \ref{RmkNonIntersectionGibbs} for some subtleties. Our computations, based on residue calculus considerations (by virtue of non-intersection we can pick out individual poles), in following the scheme in display \eqref{Scheme} are inspired by the remarkable work \cite{HuangZhang} which realises the Airy line ensemble \cite{CorwinHammond1} as the poles of a stochastic evolution on Nevanlinna functions. Alternatively we could have derived the desired equations as a limit of the finite-dimensional dynamics as we did in \cite{AM24} for the Rider-Valk\'{o} model \eqref{RVmodel}. The same argument could be applied here, see Section 4 therein. However, deriving them already in the limit directly from the bigger object $(\mathfrak{rv}_t)_{t\ge 0}$  is conceptually more attractive and showcases the usefulness of having an equation like \eqref{GLequation}.  

\section{Infinite-dimensional dynamics from Brownian motion on $\mathsf{GL}_N(\mathbb{C})$}

\subsection{Convergence on path space and non-intersection} \label{SubsectionConvPathSpace}

In this subsection we apply our general convergence framework from \cite{AM24} to the model \eqref{GLinitial}. The material here is in some sense standard as virtually identical arguments, for the individual intermediate results, have appeared before and so we will be brief. The main contribution is the non-trivial observation that this framework can also be applied to Brownian motion on $\mathsf{GL}_N(\mathbb{C})$.

We need some standard preliminaries. We refer to Section 2 of \cite{AM24} for more details. The parameter $\theta \in \mathbb{R}$ will be fixed throughout this section and will be suppressed from the notation. Consider the following geometric Brownian motion $(\boldsymbol{g}_N(t))_{t\ge 0}$, the unique strong solution in $\mathbb{R}_+$ of the SDE
\begin{equation}\label{SDEgeomBM}
\mathrm{d}\boldsymbol{g}_N(t)=\boldsymbol{g}_N(t)\mathrm{d}\mathsf{w}(t)+\left(1+\frac{\theta}{2}-N\right)\boldsymbol{g}_N(t)\mathrm{d}t.
\end{equation}
Its transition  density with respect to Lebesgue measure is given by, for $t,x,y >0$,
\begin{equation}\label{GeomBMtransKernel}
\mathsf{q}_t^{(N)}(x,y)=\frac{1}{y\sqrt{2\pi t}} \exp\left(-\frac{\left(\log y-\log x-\left(\frac{1+\theta}{2}-N\right)t\right)^2}{2t}\right).
\end{equation}

We collect the following basic facts about the model \eqref{GLinitial}.

\begin{prop}\label{PropFiniteDimProperties}
Let $\theta \in \mathbb{R}$. The SDE \eqref{GLinitial} has a unique strong solution  from any initial condition in $\mathbf{W}_{N;+}$. Moreover, for $\mathsf{x}^{(N)}(0)\in \mathbf{W}_{N;+}^\circ$ we have almost surely for all $t\ge 0$, $\mathsf{x}^{(N)}(t)\in \mathbf{W}_{N;+}^\circ$. Furthermore, the semigroup $(\mathsf{P}_{t;N})_{t\ge0}$ associated to $\mathsf{x}^{(N)}$ is $\mathrm{C}_0$-Feller continuous on $\mathbf{W}_{N;+}$ and the space of smooth symmetric  functions with compact support on $\mathbb{R}^N$ restricted to $\mathbf{W}_{N;+}$ forms a core for the generator $\mathsf{L}_N$ of $(\mathsf{P}_{t;N})_{t\ge0}$ and is moreover invariant under the action of $\mathsf{L}_N$. Finally, the transition kernel of $\mathsf{P}_{t;N}$ is given explicitly by, for $t>0, \mathbf{x}\in \mathbf{W}_{N;+}^\circ, \mathbf{y}\in \mathbf{W}_{N;+}$,
\begin{equation}\label{TransitionKernel}
\mathsf{P}_{t;N}(\mathbf{x},\mathrm{d}\mathbf{y})=\mathrm{e}^{-\lambda_Nt}\frac{\prod_{1\le i <j \le N}(y_i-y_j)}{\prod_{1\le i <j \le N}(x_i-x_j)}\det\left(\mathsf{q}_t^{(N)}(x_i,y_j)\right)_{i,j=1}^N\mathrm{d}\mathbf{y},
\end{equation}
where $\lambda_N=N(N-1)(3\theta+2-4N)/12$.
\end{prop}

\begin{proof}
 This is rather standard and an identical result (modulo a different one-dimensional transition kernel in place of $\mathsf{q}_t^{(N)}$ and constant $\lambda_N$) and proof holds for the Rider-Valk\'{o} model \eqref{RVmodel}, see Section 3 of \cite{AM24} for details. The only model specific computation is the following Doob $h$-transform structure of the generator:
 \begin{align*}
 [\mathsf{L}_Nf](x_1,\dots,x_N)&=\prod_{1\le i<j\le N}(x_i-x_j)^{-1}\left[\left(\sum_{i=1}^N\mathsf{G}_{x_i}^{(N)}\right) \left(f(x_1,\dots,x_N)\prod_{1\le i<j\le N}(x_i-x_j)\right)\right]\\
 &-\lambda_N f(x_1,\dots,x_N),
 \end{align*}
 where $\mathsf{G}_x^{(N)}=\frac{x^2}{2}\partial_x^2+(1+\frac{\theta}{2}-N)\partial_x$ is the generator of \eqref{SDEgeomBM}.
\end{proof}

\begin{rmk}
By virtue of the connection of \eqref{XiDef} to (radial) Heckman-Opdam processes \cite{Heckman-Opdam} the proposition above could in principle be derived by translating the results of \cite{Heckman-Opdam} to this setting after exponentiation of the coordinates.
\end{rmk}

Moving on, for each $N \in \mathbb{N}$, define the corners maps $\mathsf{Corners}_N^{N+1}$ from $(N+1)\times (N+1)$ matrices to $N \times N$ matrices by $\mathsf{Corners}_{N}^{N+1}([\mathbf{H}_{ij}]_{i,j=1}^{N+1})=[\mathbf{H}_{ij}]_{i,j=1}^N$. Moreover, for $N \in\mathbb{N}$ and $\mathbf{x}\in \mathbf{W}_{N;+}$ we let $\mathrm{diag}(\mathbf{x})$ be the $N \times N$  diagonal matrix with entries $\mathbf{x}$. Define the Markov kernels $\boldsymbol{\Lambda}_N^{N+1}(\mathbf{x},\mathrm{d}\mathbf{y})$ from $\mathbf{W}_{N+1;+}$ to $\mathbf{W}_{N;+}$ as the law of the eigenvalues of the $N \times N$ matrix $\mathsf{Corners}_N^{N+1}(\mathbf{U}^\dag \mathrm{diag}(\mathbf{x})\mathbf{U})$, where $\mathbf{U}$ is a random Haar distributed $N \times N$ unitary matrix. It is well-known that for $\mathbf{x}\in \mathbf{W}_{N+1;+}^\circ$, $\boldsymbol{\Lambda}_{N}^{N+1}(\mathbf{x},\mathrm{d}\mathbf{y})$ has the following explicit expression
\begin{equation}\label{ExplicitMarkovKernel}
\boldsymbol{\Lambda}_{N}^{N+1}(\mathbf{x},\mathrm{d}\mathbf{y})=\frac{N!\prod_{1\le i <j \le N}(y_i-y_j)}{\prod_{1\le i<j \le N+1}(x_i-x_j)}\mathbf{1}_{\mathbf{y}\prec \mathbf{x}} \mathrm{d}\mathbf{y},
\end{equation}
where $\mathbf{y}\prec \mathbf{x}$ denotes interlacing $x_1\ge y_1 \ge x_2 \ge y_2 \ge \cdots \ge y_N \ge x_{N+1}$.

With all this notation in place we can state the key consistency relation of the dynamics \eqref{GLinitial} as $N$ varies.

\begin{prop}\label{PropConsistency}
Let $\theta \in \mathbb{R}$. Then, we have the intertwining,
\begin{equation}\label{Intertwining}
\mathsf{P}_{t;N+1}\boldsymbol{\Lambda}_N^{N+1}=\boldsymbol{\Lambda}_N^{N+1}\mathsf{P}_{t;N}, \ \forall t \ge 0, \ \forall N \in \mathbb{N}.
\end{equation}
\end{prop}

\begin{proof}
This is implicit in \cite{Interlacing} and could also be derived from intertwining results \cite{OConnellRSK} for Brownian motions with drifts after a change of variables. The most direct way to prove it is the following: use the explicit determinant expressions \eqref{TransitionKernel} and  \eqref{ExplicitMarkovKernel} for $\mathsf{P}_{t;N}, \mathsf{P}_{t;N+1}$ and $\boldsymbol{\Lambda}_N^{N+1}$ (valid for non-coinciding coordinates, then extend by continuity using the Feller property; for the Feller property of $\boldsymbol{\Lambda}_N^{N+1}$ see \cite{AM24}), a determinant representation for $\mathbf{1}_{\mathbf{y}\prec \mathbf{x}}$ \cite{Warren}, apply the Andreief identity for integrals of products of determinants and some elementary linear algebraic manipulations to both sides of \eqref{Intertwining} to see that they are equal. This argument is presented in detail in Section 3 of \cite{AM24} and works the same here. The crux and only model-dependent input is a certain one-dimensional identity, which in this case is, for $t,x,y>0$,
\begin{equation*}
\mathsf{q}_t^{(N)}(x,y)=-\mathrm{e}^{-(\theta/2-N)t}\int_y^\infty \partial_x\mathsf{q}_t^{(N+1)}(x,z)\mathrm{d}z,
\end{equation*}
and also observing $\lambda_{N+1}-\lambda_N=N(\theta/2-N)$. This identity can be checked directly from the explicit expression \eqref{GeomBMtransKernel} for $\mathsf{q}_t^{(N)}$. Finally, \eqref{Intertwining} could alternatively be proven by deriving a consistency under $\mathsf{Corners}_{N}^{N+1}$-projection at the level of matrix processes and this implication is also explained in Section 3 of \cite{AM24}.
\end{proof}

These are the inputs we need to derive the following.

\begin{prop}\label{PropConvergenceOnPathSpace}
Let $\theta\in \mathbb{R}$. Then, there exists a unique Feller diffusion process $(\mathbf{X}_{GL}^t)_{t\ge 0}$ on $\Upsilon_+$ such that whenever, 
\begin{equation*}
\left((N^{-1}\mathsf{x}_i^{(N)}(0))_{i\in \mathbb{N}},N^{-1}\sum_{i=1}^\infty \mathsf{x}_i^{(N)}(0)\right)  \overset{N \to \infty}{\longrightarrow} \boldsymbol{v}=((x_i)_{i\in \mathbb{N}},\gamma)\in \Upsilon_+,
\end{equation*}
we have, as $N \to \infty$,
\begin{equation*}
\left((N^{-1}\mathsf{x}_i^{(N)}(\bullet))_{i\in \mathbb{N}},N^{-1}\sum_{i=1}^\infty \mathsf{x}_i^{(N)}(\bullet)\right) \overset{\mathrm{d}}{\longrightarrow} \left((\mathsf{x}_i(\bullet))_{i\in \mathbb{N}},\boldsymbol{\gamma}(\bullet)\right)= \mathbf{X}_{GL}^\bullet
\end{equation*}
in $\mathrm{C}(\mathbb{R}_+;\Upsilon_+)$
with $\mathbf{X}_{GL}^0=\boldsymbol{v}$.
\end{prop}

\begin{proof}
This immediately follows by virtue of Proposition \ref{PropFiniteDimProperties} and Proposition \ref{PropConsistency} using the general convergence theory for consistent processes with respect to the kernels $(\boldsymbol{\Lambda}_N^{N+1})_{N\in \mathbb{N}}$ from Section 2 of \cite{AM24}.
\end{proof}

Finally, we have that for non-coinciding initial conditions the limiting paths remain non-intersecting for all times.

\begin{prop}\label{NonCollisionProp}
Let $\theta \in \mathbb{R}$. Let $\boldsymbol{v}=(\mathbf{x},\gamma)\in \Upsilon_+$ with $\mathbf{x}=(x_i)_{i\in \mathbb{N}}\in \mathbf{W}_{\infty;+}^\circ$. Then, the projection of $\mathbf{X}_{GL}^\bullet$ on $(\mathsf{x}_i(\bullet))_{i\in \mathbb{N}}$ satisfies almost surely for all $t\ge 0$, $(\mathsf{x}_i(t))_{i \in \mathbb{N}}\in \mathbf{W}_{\infty;+}^\circ$.
\end{prop}

\begin{proof}
The proof is word-for-word the same as the one given in \cite{AM24} for the Rider-Valk\'{o} model \eqref{RVmodel}, which uses a Lyapunov-type argument, with the Lyapunov functions for \eqref{GLinitial} being exactly the same as the 
ones in Section 4 of \cite{AM24}.
\end{proof}

\begin{rmk}\label{RmkGibbsInitial}
Had we known a-priori that for any fixed time $t \ge 0$ almost surely $(\mathsf{x}_i(t))_{i \in \mathbb{N}}\in \mathbf{W}_{\infty;+}^\circ$ then the statement of Proposition \ref{NonCollisionProp} would also independently follow from the Gibbs property in Proposition \ref{PropGibbsPropertyGL}, see Remark \ref{RmkNonIntersectionGibbs}.
\end{rmk}

\subsection{Gibbs property}\label{SubsectionGibbs}

In this section we prove that the paths of $(\mathsf{x}_i(\bullet))_{i\in \mathbb{N}}$ from general initial condition satisfy a Gibbs resampling property, see the seminal work \cite{CorwinHammond1}, with respect to the law of exponential Brownian bridges. It is more convenient to work in logarithmic coordinates and we will equivalently prove that $(\log \mathsf{x}_i(\bullet))_{i\in \mathbb{N}}$ has the Brownian Gibbs property \cite{CorwinHammond1}.

The results and arguments in this section are not particularly surprising. However, the general framework of considering Gibbs resampling properties with respect to essentially arbitrary diffusion bridges in a unified way should be useful for future applications to other examples of random matrix dynamics.

We will work with line-ensembles with paths taking values in $\mathbb{R}$ for simplicity but the arguments below can be extended to any interval $(\mathfrak{a},\mathfrak{b})\subseteq \mathbb{R}$. So let us define,
\begin{align*}
\mathbf{W}_{\infty;+}=\left\{\mathbf{x}=(x_i)_{i\in \mathbb{N}}\in \mathbb{R}^{\mathbb{N}}:x_1\ge x_2\ge x_3 \ge \cdots \ge 0\right\},\\
\mathbf{W}_{\infty;+}^\circ=\left\{\mathbf{x}=(x_i)_{i\in \mathbb{N}}\in \mathbb{R}^{\mathbb{N}}:x_1> x_2> x_3 > \cdots >0\right\}.
\end{align*}
The basic object will be a time-homogeneous one-dimensional diffusion process in $\mathbb{R}$ with transition probability density with respect to Lebesgue measure given by $\mathfrak{p}_t(x,y)$ that we call the $\mathscr{A}$-diffusion. We assume that for any interval $[a,b]\subset \mathbb{R}_+$ and endpoints $x,y \in \mathbb{R}$, the diffusion bridge of the $\mathscr{A}$-diffusion from $x$ at time $a$ to $y$ at time $b$ exists, see \cite{MarkovianBridgers}. We call this the $\mathsf{A}$-bridge. The most natural example is simply the Brownian bridge. %Its transition probabilities $\mathfrak{pbr}_{s,t}^{a,b;x,y}(z_1,z_2)$ are given by, for $s\le t \in [a,b]$,
%\begin{equation*}
%\mathfrak{pbr}_{s,t}^{a,b;x,y}(z_1,z_2)=\frac{\mathfrak{p}_{t-s}(z_1,z_2)\mathfrak{p}_{b-t}(z_2,y)}{\mathfrak{p}_{b-s}(z_1,y)}, \ \ z_1, z_2 \in \mathbb{R}.
%\end{equation*}

\begin{defn}
Let $[a,b]\subset\mathbb{R}_+$, $k\in\mathbb{N}$, and $\mathbf{x},\mathbf{y}\in\mathbf{W}_{k}$.
Denote by
$\mathbf{P}^{a,b;\mathbf{x},\mathbf{y}}$ the law of $k$ independent $\mathsf{A}$-bridges
started at $\mathbf{x}$ at time $a$ and ending at $\mathbf{y}$ at time $b$ and write 
$\mathbf{E}^{a,b;\mathbf{x},\mathbf{y}}$ for the expectation with respect to
$\mathbf{P}^{a,b;\mathbf{x},\mathbf{y}}$.
\end{defn}

Observe that, dependence on $k\in \mathbb{N}$ of $\mathbf{P}^{a,b;\mathbf{x},\mathbf{y}}$ is implicit in the length of $\mathbf{x} \in \mathbf{W}_k$. To ease notation write $\mathrm{C}(\mathbb{R}_+)=\mathrm{C}(\mathbb{R}_+;\mathbb{R})$ (analogously for an interval $[a,b]\subset \mathbb{R}_+$) and $\mathrm{C}(\llbracket 1, N \rrbracket \times \mathbb{R}_+)=\prod_{i\in \llbracket 1, N \rrbracket } \mathrm{C}(\mathbb{R}_+)$ where $N\in \mathbb{N}\cup \{\infty\}$, with $\llbracket 1, \infty \rrbracket\equiv\mathbb{N}$. One could also replace the target space $\mathbb{R}$ with a general connected open interval.

\begin{defn}
A line ensemble is a $\mathrm{C}(\llbracket 1, N \rrbracket \times \mathbb{R}_+)$-valued random variable.
\end{defn}

We will need the following assumptions on the $\mathsf{A}$-bridge which we enforce throughout this section. We them to be true quite generally. It is well-known that Brownian bridge satisfies these, see \cite{CorwinHammond1}. The Bessel bridge does so as well (in $(0,\infty)$ instead of $\mathbb{R}$), see \cite{XuanWuBessel}.

\begin{ass}\label{ass-bridge}
Let $[a,b]\in\mathbb{R}_+$ and $x,y\in \mathbb{R}$ be arbitrary. We assume that the 
    $\mathbf{P}^{a,b;x,y}$-distributed random variable $\mathsf{A}$ satisfies:
\begin{itemize}
    \item [1.] 
    For any $f,g \in \mathrm{C}\left([a,b]\right)$ with $f(a)>x,f(b)>y$ and $g(a)<x, g(b)<y$, 
\begin{align}\label{noTouch}
    \mathbf{P}^{a,b;x,y}\left(\left\{\inf_{t\in[a,b]}\left(f(t)-\mathsf{A}(t)\right) = 0\right\}\cap
    \left\{\inf_{t\in[a,b]}\left(\mathsf{A}(t)- g(t)\right) = 0\right\}\right) = 0.
\end{align}
\item [2.]
For any open set \(\mathscr{O} \subset \mathrm{C}\left([a,b]\right)\) containing \(f\) with \(f(a) = x\) and \(f(b) = y\), we have
\[
\mathbf{P}^{a,b;x,y}\left(\mathsf{A} \in\mathscr{O}  \right) > 0.
\]

\end{itemize}

\end{ass}

By virtue of Assumption \ref{ass-bridge} we can define.

\begin{defn}\label{DefAvoidingBridges}
Let $[a,b]\subset \mathbb{R}_+$, $k\in \mathbb{N}$ and $\mathbf{x}=(x_i)_{i=1}^k$, $\mathbf{y}=(y_i)_{i=1}^k\in \mathbf{W}_{k}^\circ$. Let $u, l \in \mathrm{C}([a,b])$ be such that: $u(t)>l(t)$ for all $t\in [a,b]$ and $u(a)>x_1, u(b)>y_1, l(a)<x_k, l(b)<y_k$. We define $\mathbf{P}_{\mathrm{avoid}}^{a,b;\mathbf{x},\mathbf{y};u,l}$ to be the law of the process $(\mathsf{A}(t;i))_{i\in \llbracket 1, k \rrbracket,t\in [a,b]}$ distributed according to $\mathbf{P}^{a,b;\mathbf{x},\mathbf{y}}$ and conditioned on the non-zero probability event
\begin{equation*}
\left\{u(r)>\mathsf{A}(r;1)>\mathsf{A}(r;2)>\cdots >\mathsf{A}(r;k)>l(r), \ \ \forall r\in [a,b]\right\}.
\end{equation*}
Write $\mathbf{E}_{\mathrm{avoid}}^{a,b;\mathbf{x},\mathbf{y};u,l}$ for the corresponding expectation.
\end{defn}

We call $\mathbf{x}, \mathbf{y},u,l$ from Definition \ref{DefAvoidingBridges} the boundary data.

We can now define the Gibbs resampling property. It will be convenient to write $\mathsf{Law}(\mathsf{X})$ for the law of a random variable $\mathsf{X}$ and $\mathsf{Law}(\mathsf{X}|\mathscr{F})$ for its conditional law given a sigma algebra $\mathscr{F}$.

\begin{defn}\label{GibbsDefinition}
Let $N \in \mathbb{N}\cup \{\infty\}$. We say that a line ensemble $\mathcal{L}\in \mathrm{C}(\llbracket 1, N \rrbracket)$ satisfies the Gibbs resampling property with respect to the law of the $\mathsf{A}$-bridge if for all $\llbracket \ell,\ell+k-1 \rrbracket \subset \llbracket 1, N \rrbracket$ with $\ell, k \ge 1$ and any $[a,b]\in \mathbb{R}_+$ we have
\begin{equation}\label{GibbsPropertyFormal}
\mathsf{Law}\left(\mathcal{L}|_{\llbracket \ell,\ell+k-1 \rrbracket \times [a,b]}|\mathscr{F}_{\mathrm{ext}}\left(\llbracket \ell, \ell+k-1 \rrbracket \times (a,b) \right)\right)=\mathbf{P}_{\mathrm{avoid}}^{a,b;\mathbf{x},\mathbf{y};u,l},\ \ \textnormal{ almost surely,}
\end{equation}
where the external sigma-algebra $\mathscr{F}_{\mathrm{ext}}\left(\llbracket \ell, \ell+k-1 \rrbracket \times (a,b) \right)$ is defined as
\begin{equation}\label{ExternalSigmaAlgebra}
\mathscr{F}_{\mathrm{ext}}\left(\llbracket \ell, \ell+k-1 \rrbracket \times (a,b) \right)= \sigma\left(\mathcal{L}_j(t);(j,t)\notin \llbracket \ell, \ell+k-1 \rrbracket\times (a,b)\right), 
\end{equation}
and the (random) boundary data $\mathbf{x},\mathbf{y},u,l$ given by,
\begin{align*}
\mathbf{x}=\left(\mathcal{L}_\ell(a),\mathcal{L}_{\ell+1}(a),\dots,\mathcal{L}_{k+\ell-1}(a)\right), \  \mathbf{y}=\left(\mathcal{L}_{\ell}(b),\mathcal{L}_{\ell+1}(b),\dots,\mathcal{L}_{k+\ell-1}(b)\right), \ u=\mathcal{L}_{\ell-1},  \ l=\mathcal{L}_{\ell+k},
\end{align*}
with the convention $\mathcal{L}_{0}\equiv \infty$, $\mathcal{L}_{N+1}\equiv -\infty$.
\end{defn}

\begin{rmk}\label{RmkNonIntersectionGibbs}
Observe that, in our definition of the Gibbs property it is implicit that for all fixed $t\ge 0$, almost surely $(\mathcal{L}_i(t))_{i\in \llbracket 1,N \rrbracket }\in \mathbf{W}_N^\circ$ in order for the right hand side of \eqref{GibbsPropertyFormal} to be well-defined. In fact, a stronger statement holds: almost surely for all $t\ge 0$ we have $(\mathcal{L}_i(t))_{i\in \llbracket 1,N \rrbracket }\in \mathbf{W}_N^\circ$ as we spell out now. Write $\mathbf{P}$ and $\mathbf{E}$ for $\mathsf{Law}(\mathcal{L})$ and expectation with respect to it. It suffices to show that for all $k \in \mathbb{N}$ and times $T\ge 0$ we have,
\begin{equation*}
\mathbf{P}\left(\inf_{t\in [0,T]}\left|\mathcal{L}_k(t)-\mathcal{L}_{k+1}(t)\right|>0\right)=1.
\end{equation*}
From \eqref{GibbsPropertyFormal} we have for any bounded measurable function $\mathsf{F}:\mathrm{C}(\{k,k+1\}\times[0,T])\to \mathbb{R}_+$,
\begin{equation*}
\mathbf{E}\left[\mathsf{F}\left(\mathcal{L}|_{\{k,k+1\}\times[0,T]}\right)\big|\mathscr{F}_{\mathrm{ext}}\left(\{k,k+1\} \times (0,T) \right)\right]=\mathbf{E}_{\mathrm{avoid}}^{0,T;\mathbf{x},\mathbf{y};u,l}\left[\mathsf{F}(\mathsf{Z})\right], \ \textnormal{ almost surely},
\end{equation*}
where $\mathsf{Z}$ is distributed according to $\mathbf{P}_{\mathrm{avoid}}^{0,T;\mathbf{x},\mathbf{y};u,l}$, and $\mathbf{x},\mathbf{y},u,l$ as in Definition \ref{GibbsDefinition}. Taking expectations we get,
\begin{equation*}
\mathbf{E}\left[\mathsf{F}\left(\mathcal{L}|_{\{k,k+1\}\times[0,T]}\right)\right]=\mathbf{E}\left[\mathbf{E}_{\mathrm{avoid}}^{0,T;\mathbf{x},\mathbf{y};u,l}\left[\mathsf{F}(\mathsf{Z})\right]\right].
\end{equation*}
 Choosing $\mathsf{F}(f,g)=\mathbf{1}_{f(t)>g(t), \forall t \in [0,T]}$ we get the desired statement since the inner expectation is 1.
\end{rmk}

We next point out that finite-dimensional line ensembles arising from a Doob $h$-transform of a Karlin-McGregor semigroup \cite{KarlinMcGregor} have a Gibbs property. %This is done directly at the continuum level and avoids discrete approximations as in previous works.
Namely, assume there exists an $N$-dimensional Markov process $(\mathsf{v}^{(N)}(t))_{t\ge 0}$ in $\mathbf{W}_{N}^\circ$ with continuous sample paths with transition probabilities given by the following Doob $h$-transform of $\mathscr{A}$-diffusions,
\begin{align}
    \label{Doob-transitioDensity}
	\mathfrak{P}_N^{(\mathrm{h})}(t)(\mathbf{x},\mathbf{y})= \mathrm{e}^{-\mu_Nt}\frac{\mathrm{h}(\mathbf{y})}{\mathrm{h}(\mathbf{x})} 
    \det\left(\mathfrak{p}_t(x_i,y_j)\right)_{i,j=1}^{N}, \  \ \mathbf{x} \in \mathbf{W}_N^{\circ}, \mathbf{y}\in\mathbf{W}_N^\circ,  \ t>0,
\end{align}
for some function $\mathrm{h}$ strictly positive in $\mathbf{W}_N^\circ$ and $\mu_N\in \mathbb{R}$. We note that, the process $(\mathsf{v}^{(N)}(t))_{t\ge 0}$ gives rise to a $\mathrm{C}(\llbracket 1, N \rrbracket \times \mathbb{R}_+)$-valued line ensemble.

\begin{prop}\label{Gibbs-finiteN}
    Let $N\in\mathbb{N}$. Then, the line ensemble associated to the process $(\mathsf{v}^{(N)}(t))_{t\ge 0}$ satisfies the Gibbs resampling property with respect to the law of the $\mathsf{A}$-bridge. 
\end{prop}

\begin{proof}
    Let us take an arbitrary
    interval $[a,b]\subset\mathbb{R}_+$ partitioned as follows
\begin{equation*}
   a= t_0< t_1< \cdots< t_k= b.
\end{equation*}
Denote by $\mathbf{x}^{(m)} = (x_1^{(m)}, \dots, x_N^{(m)})\in\mathbf{W}_{N}$ the positions
at time $t_m$, for $m=0, 1, \dots,k$,
with
$\mathbf{x}^{(0)} = \mathbf{x}=
\left(x_1, \dots, x_N\right)$, $\mathbf{x}^{(m)} = \mathbf{y}
=\left(y_1, \dots, y_N\right)$. Applying the Markov property, the corresponding finite-dimensional distribution of 
$\mathsf{v}^{(N)}(\bullet)$ conditional on $\mathsf{v}^{(N)}(a)=\mathbf{x}$ and $\mathsf{v}^{(N)}(b)=\mathbf{y}$
is given by the following joint density function with respect to product Lebesgue measure,
\begin{align*}
\frac{
\prod_{m=0}^{k-1} \mathfrak{P}^{(h)}_{t_{m+1} - t_m}\left(\mathbf{x}^{(m)}, \mathbf{x}^{(m+1)}\right)
}{
\mathfrak{P}^{(h)}_{b-a}\left(\mathbf{x}, \mathbf{y}\right)
}, 
\end{align*}
which, using \eqref{Doob-transitioDensity} and after  cancellation of common multiplicative terms, simplifies to
\begin{align}
\frac{
\prod_{m=0}^{k-1} \det\left(\mathfrak{p}_{t_{m+1}-t_m}\left(x_i^{(m)},x_j^{(m+1)}\right)\right)_{i,j=1}^{N}}{
\det\left(\mathfrak{p}_{b-a}(x_i,y_j)\right)_{i,j=1}^{N}
}.
\end{align}
Note that, this is exactly the density function corresponding to the finite-dimensional distribution of $N$ non-colliding
$\mathsf{A}$-bridges
on the interval $[a,b]$ with endpoints $\mathbf{x}$ and $ \mathbf{y}$, as obtained by the Karlin–McGregor formula.
Hence, by virtue of continuity of the paths, the conditional law of $\mathsf{v}^{(N)}(\bullet)$ on any interval $[a,b]$, given $\mathsf{v}^{(N)}(a)$ and $\mathsf{v}^{(N)}(b)$, coincides with the law of $N$ non-intersecting $\mathsf{A}$-bridges on $[a,b]$, with the same endpoints. It is then obvious (but tedious to write down formally and we omit this) that non-intersecting $\mathsf{A}$-bridges satisfy the Gibbs resampling property with respect to the law of the $\mathsf{A}$-bridge which completes the proof.
\end{proof}

We require one final lemma.

\begin{lem}\label{Lem-CondDistConv}
Assume that for all $[a,b] \subset \mathbb{R}_+$, and $x_N,y_N \in \mathbb{R}$ such that $x_N \to x$, $y_N \to y$ as $N \to \infty$, we have the weak convergence of probability measures,
\begin{equation}\label{SingleBridgeConv}
\mathbf{P}_N^{a,b;x_N,y_N} \longrightarrow \mathbf{P}^{a,b;x,y}, \ \textnormal{ as } N \longrightarrow \infty,
\end{equation}
where $\mathbf{P}_N^{a,b;x_N,y_N}$ corresponds to the law of a $\mathsf{A}_N$-bridge and $\mathbf{P}^{a,b;x,y}$ to the law of a $\mathsf{A}$-bridge.  Let $k\in \mathbb{N}$ and $\mathbf{x}_N,\mathbf{y}_N, \mathbf{x},\mathbf{y}\in \mathbf{W}_{k}^\circ$ with $\mathbf{x}_N \to \mathbf{x}, \mathbf{y}_N \to \mathbf{y}$ as $N \to \infty$. Moreover, suppose that $u_N,u \in \mathrm{C}([a,b];\mathbb{R}\cup \{ \infty \})$ and $l_N,l \in \mathrm{C}([a,b];\mathbb{R}\cup \{ -\infty \})$ are such that 
\begin{equation*}
u(t)>l(t), \ \forall t \in [a,b], \ u(a)>x_1, \ u(b)>y_1, \ l(a)<x_k, \ l(b)<y_k,
\end{equation*}
and also, as $N \to \infty$, uniformly in $[a,b]$,
\begin{equation}\label{BoundaryConvergence}
u_N \longrightarrow u, \ \ l_N \longrightarrow l.
\end{equation}
Then, for $N \in \mathbb{N}$ large enough, the conditional law $\mathbf{P}_{\mathrm{avoid};N}^{a,b;\mathbf{x}_N,\mathbf{y}_N;u_N,l_N}$ is well-defined and we have the following weak convergence of probability measures, as $N \longrightarrow \infty$,
\begin{equation}\label{ConditionedBridgeConv}
\mathbf{P}_{\mathrm{avoid};N}^{a,b;\mathbf{x}_N,\mathbf{y}_N;u_N,l_N}\longrightarrow \mathbf{P}_{\mathrm{avoid}}^{a,b;\mathbf{x},\mathbf{y};u,l}.
\end{equation}
\end{lem}

\begin{proof}
Observe that, by virtue of Assumption \ref{ass-bridge}, the following event
\begin{equation*}
\mathsf{Ord}= \left\{u(t)>\mathsf{A}(t;1)>\mathsf{A}(t;2)>\cdots >\mathsf{A}(t;k)>l(t), \ \forall t \in [a,b]\right\},
\end{equation*}
where $(\mathsf{A}(t;i))_{i\in \llbracket 1, k\rrbracket,t \in [a,b]}$ is distributed according to $\mathbf{P}^{a,b;\mathbf{x},\mathbf{y}}$ ($k$ independent $\mathsf{A}$-bridges) has positive probability and for all $N \in \mathbb{N}$ also consider the event $\mathsf{Ord}_N$ given by,
\begin{equation}
\mathsf{Ord}_N=\left\{u_N(t)>\mathsf{A}_N(t;1)>\mathsf{A}_N(t;2)>\cdots >\mathsf{A}_N(t;k)>l_N(t), \ \forall t \in [a,b]\right\},
\end{equation}
where $(\mathsf{A}_N(t;i))_{i\in \llbracket 1, k\rrbracket,t \in [a,b]}$ is distributed according to $\mathbf{P}_N^{a,b;\mathbf{x},\mathbf{y}}$
($k$ independent $\mathsf{A}_N$-bridges). On the other hand, from \eqref{SingleBridgeConv} and by independence of the coordinates we have
\begin{equation}\label{MultipleBridgeConv}
 (\mathsf{A}_N(t;i))_{i\in \llbracket 1, k\rrbracket,t \in [a,b]} \overset{\mathrm{d}}\longrightarrow  (\mathsf{A}(t;i))_{i\in \llbracket 1, k\rrbracket,t \in [a,b]}.
\end{equation}
By virtue of Skorokhod's representation theorem we can assume this convergence takes place almost surely in a coupling given by a probability measure $\mathbf{Q}$. Hence, also recalling \eqref{BoundaryConvergence}, we obtain that for $N$ large enough $\mathbf{Q}(\mathsf{Ord}_N)>0$ and  $\mathbf{P}_{\mathrm{avoid};N}^{a,b;\mathbf{x}_N,\mathbf{y}_N;u_N,l_N}$ is well-defined. To finally prove \eqref{ConditionedBridgeConv} we need to show that for any continuous bounded functional $\mathsf{F}:\mathrm{C}(\llbracket 1, k \rrbracket \times [a,b])\to \mathbb{R}_+$,
\begin{equation*}
\frac{\mathbf{Q}\left[\mathsf{F}\left((\mathsf{A}_N(t;i))_{i\in \llbracket 1, k\rrbracket,t \in [a,b]} \right)\mathbf{1}_{\mathsf{Ord}_N}\right]}{\mathbf{Q}\left(\mathsf{Ord}_N\right)} \longrightarrow \frac{\mathbf{Q}\left[\mathsf{F}\left((\mathsf{A}(t;i))_{i\in \llbracket 1, k\rrbracket,t \in [a,b]} \right)\mathbf{1}_{\mathsf{Ord}}\right]}{\mathbf{Q}\left(\mathsf{Ord}\right)}, \ \textnormal{ as } N \to \infty.
\end{equation*}
This convergence then follows from \eqref{MultipleBridgeConv} by virtue of the bounded convergence theorem and the fact that the topological boundary of the set $\mathsf{Ord}$ has probability zero by virtue of Assumption \ref{ass-bridge}.
\end{proof}

We have the following natural result.

\begin{prop}\label{PropositionLimitingGibbs}
Let $(\mathcal{L}^{(N)})_{N\in\mathbb{N}} $ be a sequence of line ensembles such that for each $N\in\mathbb{N}$, with $\mathcal{L}^{(N)}\in \mathrm{C}(\llbracket 1, N \rrbracket \times \mathbb{R}_+)$  satisfies the Gibbs resampling property with respect to the law of a $\mathsf{A}_N$-bridge. Let $\mathcal{L}$ be a $\mathrm{C}(\mathbb{N}\times\mathbb{R}_+)$-valued line ensemble. Moreover, assume that as $N \longrightarrow \infty$, for all $M \in \mathbb{N}$ (when $N\ge M$),
\begin{align*} \left(\mathcal{L}^{(N)}_{i}(t)\right)_{i\in \llbracket 1, M \rrbracket, t\in \mathbb{R}_+}\overset{\mathrm{d}}{\longrightarrow}\left(\mathcal{L}_{i}(t)\right)_{i\in \llbracket 1, M \rrbracket, t\in \mathbb{R}_+} \ \textnormal{ in }\; \mathrm{C}\left(\llbracket 1, M \rrbracket \times\mathbb{R}_+\right), 
\end{align*}
and that for any fixed $t\ge 0$, almost surely $(\mathcal{L}_i(t))_{i\in \mathbb{N}}\in \mathbf{W}_\infty^\circ$. Finally, we assume that for any $[a,b]\subset \mathbb{R}_+$ and any $x_N,y_N \in \mathbb{R}$, whenever $x_N \to x$ and $y_N \to y$, we have the weak convergence of probability measures
\begin{equation*}
\mathbf{P}_N^{a,b;x_N,y_N} \longrightarrow\mathbf{P}^{a,b;x,y}, \ \ \textnormal{ as } N \longrightarrow\infty,
\end{equation*}
where $\mathbf{P}_N^{a,b;x_N,y_N}$ and $\mathbf{P}^{a,b;x,y}$ correspond to the law of a $\mathsf{A}_N$-bridge and $\mathsf{A}$-bridge respectively. Then, $\mathcal{L}$ satisfies the Gibbs resampling property with respect to the law of an $\mathsf{A}$-bridge.
In particular, the paths of $\mathcal{L}$ are almost surely non-intersecting.
\end{prop}

\begin{proof}

We fix an index set
$ \llbracket \ell , \ell+k-1 \rrbracket \subset \mathbb{N}$, with $\ell,k\ge 1$,
and an arbitrary time interval \([a,b]\subset{\mathbb{R}_+}\).
Recall the definition of the external sigma-algebra $\mathscr{F}_{\mathrm{ext}}\left(\llbracket \ell, \ell+k-1 \rrbracket \times (a,b) \right)$ from \eqref{ExternalSigmaAlgebra}. To prove the result, we need to show that
for any bounded Borel-measurable function
$\mathsf{F}: \mathrm{C}\left(\llbracket \ell, \ell+k-1 \rrbracket \times [a,b]\right) \to \mathbb{R}_+$,
we have, with $\mathbf{E}$ denoting expectation with respect to $\mathsf{Law}(\mathcal{L})$,
\begin{equation}\label{GibbsGoalIdentity}
\mathbf{E}\left[ \mathsf{F}\left(\mathcal{L}|_{\llbracket \ell, \ell+k-1 \rrbracket \times [a,b]}\right) \big| \mathscr{F}_{\mathrm{ext}}\left(\llbracket \ell, \ell+k-1\rrbracket\times (a,b)\right) \right]
= \mathbf{E}_{\mathrm{avoid}}^{\mathbf{x},\mathbf{y};u,l}\left[\mathsf{F}(\mathsf{Z})\right], \  \ \textnormal{almost surely},  
\end{equation}
where $\mathsf{Z}$ has the law
$\mathbf{P}_{\textnormal{avoid}}^{\mathbf{x},\mathbf{y};u,l}$,
and, the random boundary data is given by
\begin{align*}
\mathbf{x}&=\left(\mathcal{L}_{\ell}(a),\mathcal{L}_{\ell+1}(a),\dots,\mathcal{L}_{\ell+k-1}(a)\right),\  \ \mathbf{y}=
\left(\mathcal{L}_{\ell}(b),\mathcal{L}_{\ell+1}(b),\dots,\mathcal{L}_{\ell+k-1}(b)\right),
\\ u&=
\mathcal{L}_{\ell-1}\  \  (u=+\infty, \textnormal{ if } \ell=1),\hspace{7mm}\  \ l=
\mathcal{L}_{\ell+k}.
\end{align*}
Observe that the right hand side of \eqref{GibbsGoalIdentity} is well-defined almost surely by virtue of the assumption that for any fixed $t\ge 0$, almost surely $(\mathcal{L}_i(t))_{i\in \mathbb{N}}\in \mathbf{W}_\infty^\circ$. Here and below, we omit the parameters $a,b$ from the notation for brevity.

By the definition and since
$ \mathbf{E}_{\mathrm{avoid}}^{\mathbf{x}, \mathbf{y}; u, l}\left[\mathsf{F}(\mathsf{Z})\right]
$
is  $\mathscr{F}_{\mathrm{ext}}\left(\llbracket \ell, \ell+k-1\rrbracket \times (a,b)\right)$-measurable,
it suffices to  verify that
\begin{align*}
\mathbf{E}\left[\mathbf{1}_{\mathscr{B}}\mathsf{F}\left(\mathcal{L}|_{\llbracket \ell, \ell+k-1 \rrbracket \times [a,b]}\right)\right] = \mathbf{E}\left[\mathbf{1}_{\mathscr{B}}\mathbf{E}_{\mathrm{avoid}}^{\mathbf{x},\mathbf{y};u,l}\left[\mathsf{F}(\mathsf{Z})\right]\right],
    \  \ \forall \mathscr{B}\in\mathscr{F}_{\mathrm{ext}}\left(\llbracket \ell, \ell+k-1\rrbracket \times (a,b)\right).
\end{align*}
Consider now the following $\pi$-system of cylinder sets,
\begin{align*}
\mathfrak{C} \overset{\textnormal{def}}{=}  \bigg\{ \mathscr{B}&=
\left\{\left (\mathcal{L}_{i_1}(t_1), \ldots, \mathcal{L}_{i_m}(t_m)\right) \in A \right\}: \\
&(i_j,t_j) \notin \llbracket \ell, \ell+k-1 \rrbracket \times (a,b),\,j\in\llbracket 1,m\rrbracket,\,
m\in\mathbb{N},\; A \subseteq \mathbb{R}^m \textnormal{ Borel}\bigg\}.
\end{align*}
From standard measure-theoretic results, we have 
\begin{align*}
\sigma\left(\mathfrak{C}\right)=\mathscr{F}_{\mathrm{ext}}\left(\llbracket \ell, \ell+k-1 \rrbracket \times (a,b)\right).
\end{align*}
Next, define
\begin{align*}
\mathfrak{D} \overset{\textnormal{def}}{=}  \left\{ \mathscr{B}\in \mathscr{F}_{\mathrm{ext}}\left(\llbracket \ell, \ell+k-1 \rrbracket \times (a,b)\right): \mathbf{E}\left[\mathbf{1}_{\mathscr{B}}\mathsf{F}\left(\mathcal{L}|_{\llbracket \ell, \ell+k-1\rrbracket \times [a,b]}\right)\right] = \mathbf{E}\left[\mathbf{1}_{\mathscr{B}}\, \mathbf{E}_{\mathrm{avoid}}^{\mathbf{x},\mathbf{y};u,l}\left[\mathsf{F}(\mathsf{Z})\right]\right] \right\}.
\end{align*}
We show in the following that
\begin{align}\label{ConditionalExpectationIdentity}
\mathbf{E}\left[\mathbf{1}_{\mathscr{B}}\mathsf{F}\left(\mathcal{L}|_{\llbracket \ell, \ell+k-1 \rrbracket \times [a,b]}\right)\right] = \mathbf{E}\left[\mathbf{1}_{\mathscr{B}}\mathbf{E}_{\mathrm{avoid}}^{\mathbf{x},\mathbf{y};u,l}\left[\mathsf{F}(\mathsf{Z})\right]\right],
    \  \ \forall \mathscr{B}\in\mathfrak{C}.
\end{align}
Assuming \eqref{ConditionalExpectationIdentity}, it is straightforward to verify, using the monotone convergence theorem, that \(\mathfrak{D}\) is a (Dynkin) \(\lambda\)-system containing $\mathfrak{C}$. Hence, the desired result follows by
applying the %(Dynkin) 
$\pi-\lambda$ theorem.

We now proceed to proving the identity
\eqref{ConditionalExpectationIdentity}.  To establish \eqref{ConditionalExpectationIdentity}, we begin by considering the case where $\mathsf{F}$ is continuous (and bounded). By a density argument it moreover suffices to  prove the following. For any $m \in \mathbb{N}$,  $(i_j,t_j) \notin \llbracket \ell, \ell+k-1 \rrbracket \times (a,b),\,j\in\llbracket 1,m\rrbracket$ and any continuous bounded function $\mathfrak{H}$ from $\mathbb{R}^m$ to $\mathbb{R}$ we have:
\begin{align}\label{ConditionalExpectWithFunction}
&\mathbf{E}\left[\mathfrak{H}\left (\mathcal{L}_{i_1}(t_1), \ldots, \mathcal{L}_{i_m}(t_m)\right)\mathsf{F}\left(\mathcal{L}|_{\llbracket \ell, \ell+k-1 \rrbracket \times [a,b]}\right)\right]\nonumber \\&= \mathbf{E}\left[\mathfrak{H}\left (\mathcal{L}_{i_1}(t_1), \ldots, \mathcal{L}_{i_m}(t_m)\right)\mathbf{E}_{\mathrm{avoid}}^{\mathbf{x},\mathbf{y};u,l}\left[\mathsf{F}(\mathsf{Z})\right]\right].
\end{align}
By the Skorokhod representation theorem, we realize $(\mathcal{L}^{(N)})_{N\in \mathbb{N}}$ and $\mathcal{L}$ on a common probability space $(\Omega,\mathbf{Q})$ such that, $\mathbf{Q}$-almost surely, as $N \longrightarrow \infty$, for all $M \in \mathbb{N}$ (with $N \ge M$),
\begin{align}\label{a.s.conv} \left(\mathcal{L}^{(N)}_{i}(t)\right)_{i\in \llbracket 1, M \rrbracket, t\in \mathbb{R}_+}\longrightarrow\left(\mathcal{L}_{i}(t)\right)_{i\in \llbracket 1, M \rrbracket, t\in \mathbb{R}_+} \ \textnormal{ in }\; \mathrm{C}\left(\llbracket 1, M \rrbracket \times\mathbb{R}_+\right).
\end{align}
Hence, $\mathbf{Q}$-almost surely, as $N \longrightarrow \infty$ we have,
\begin{equation}\label{ExtFnConv}
\mathfrak{H}\left (\mathcal{L}^{(N)}_{i_1}(t_1), \ldots, \mathcal{L}^{(N)}_{i_m}(t_m)\right) \longrightarrow \mathfrak{H}\left (\mathcal{L}_{i_1}(t_1), \ldots, \mathcal{L}_{i_m}(t_m)\right).
\end{equation}
Moreover, for each \(N\in\mathbb{N}\)
with $N\ge  \ell+k$, by virtue of the Gibbs property of $\mathcal{L}^{(N)}$ one obtains with $\mathfrak{H}$ as above
\begin{align}
&\mathbf{Q}\left[\mathfrak{H}\left (\mathcal{L}^{(N)}_{i_1}(t_1), \ldots, \mathcal{L}^{(N)}_{i_m}(t_m)\right)\mathsf{F}\left(\mathcal{L}^{(N)}|_{\llbracket \ell, \ell+k-1 \rrbracket \times [a,b]}\right)\right]\nonumber
\\&= \mathbf{Q}\left[\mathfrak{H}\left (\mathcal{L}^{(N)}_{i_1}(t_1), \ldots, \mathcal{L}^{(N)}_{i_m}(t_m)\right) \mathbf{E}_{N,\mathrm{avoid}}^{\mathbf{x}^{(N)},\mathbf{y}^{(N)};u^{(N)},l^{(N)}}\left[\mathsf{F}\left(\mathsf{Z}^{(N)}\right)\right]\right]\label{BGP_N},
\end{align}
where
$\mathsf{Z}^{(N)}$
is 
$\mathbf{P}_{N,\mathrm{avoid}}^{\mathbf{x}^{(N)},\mathbf{y}^{(N)};u^{(N)},l^{(N)}}$-distributed and the random boundary data is given by,
\begin{align*}
\mathbf{x}^{(N)}&=
\left(\mathcal{L}_{\ell}^{(N)}(a),\mathcal{L}_{\ell +1}^{(N)}(a),\dots,\mathcal{L}_{\ell+k-1}^{(N)}(a)\right),\  \ \mathbf{y}^{(N)}=
\left(\mathcal{L}_{\ell}^{(N)}(b),\mathcal{L}_{\ell+1}^{(N)}(b),\dots,\mathcal{L}_{\ell+k-1}^{(N)}(b)\right),
\\ u^{(N)}&=
\mathcal{L}_{\ell-1}^{(N)},\  \  (u^{(N)}=+\infty, \textnormal{ if } \ell=1),\hspace{2mm}\  \ l^{(N)}=
\mathcal{L}_{\ell+k}^{(N)}.
\end{align*} 
In light of \eqref{a.s.conv}
and \eqref{ExtFnConv}, 
we have using the bounded convergence theorem  and moreover using Lemma \ref{Lem-CondDistConv} (recall that for any fixed $t\ge 0$, almost surely $(\mathcal{L}_i(t))_{i\in \mathbb{N}}\in \mathbf{W}_\infty^\circ$)
that
\begin{align}
&\mathbf{Q}\left[\mathfrak{H}\left (\mathcal{L}^{(N)}_{i_1}(t_1), \ldots, \mathcal{L}^{(N)}_{i_m}(t_m)\right) \mathsf{F}\left(\mathcal{L}^{(N)}|_{\llbracket \ell, \ell+k-1\rrbracket \times [a,b]}\right)\right]\nonumber \\
&\xrightarrow{N \to \infty}
\mathbf{Q}\left[\mathfrak{H}\left (\mathcal{L}_{i_1}(t_1), \ldots, \mathcal{L}_{i_m}(t_m)\right)\mathsf{F}\left(\mathcal{L}|_{\llbracket \ell, \ell+k-1\rrbracket \times [a,b]}\right)\right], \label{RHS-conv}
\\
&\mathbf{Q}\left[\mathfrak{H}\left (\mathcal{L}^{(N)}_{i_1}(t_1), \ldots, \mathcal{L}^{(N)}_{i_m}(t_m)\right)\mathbf{E}_{N,\mathrm{avoid}}^{\mathbf{x}^{(N)},\mathbf{y}^{(N)},u^{(N)},l^{(N)}}\left[\mathsf{F}\left(\mathsf{Z}^{(N)}\right)\right]\right]\nonumber
\\
&\xrightarrow{N \to \infty}
\mathbf{Q}\left[\mathfrak{H}\left (\mathcal{L}_{i_1}(t_1), \ldots, \mathcal{L}_{i_m}(t_m)\right)\mathbf{E}_{\mathrm{avoid}}^{\mathbf{x},\mathbf{y};u,l}[\mathsf{F}(\mathsf{Z})]\right].\label{LHS-conv}
\end{align}
Therefore, \eqref{ConditionalExpectWithFunction} and thus \eqref{ConditionalExpectationIdentity} with $\mathsf{F}$ continuous and bounded follows from
\eqref{RHS-conv} and \eqref{LHS-conv} by taking $N\to\infty$ in \eqref{BGP_N}.
Now, let us define
\begin{align*}
    \mathscr{G}\overset{\textnormal{def}}{=}  \bigg\{ &\mathsf{F} :\mathrm{C}\left(\llbracket \ell, \ell+k-1\rrbracket \times [a,b]\right)\to\mathbb{R}_+:  \\
    &\mathsf{F} \textnormal{ is a bounded, Borel-measurable function satisfying }\eqref{ConditionalExpectationIdentity}  \bigg\}.
\end{align*}
It is straightforward to verify that \(\mathscr{G}\) is a monotone class. Moreover, from the argument above, \(\mathscr{G}\) contains the algebra of bounded continuous functions on \(\mathrm{C}\left(\llbracket \ell, \ell+k-1\rrbracket \times [a,b]\right)\), which generates the Borel \(\sigma\)-algebra on this space. Therefore, by the monotone class theorem, \(\mathscr{G}\) contains all bounded Borel-measurable functions on \(\mathrm{C}\left(\llbracket \ell, \ell+k-1\rrbracket \times [a,b]\right)\) which completes the proof of \eqref{ConditionalExpectationIdentity}. The very final statement then follows from Remark \ref{RmkNonIntersectionGibbs}.
\end{proof}

We can finally prove the following.

\begin{prop}\label{PropGibbsPropertyGL}
Let $\theta \in \mathbb{R}$. Let $\boldsymbol{v}=(\mathbf{x},\gamma)\in \Upsilon_+$ with $\mathbf{x}=(x_i)_{i\in \mathbb{N}}\in \mathbf{W}_{\infty;+}^\circ$. Then, the projection of $\mathbf{X}_{GL}^\bullet$ on $(\mathsf{x}_i(\bullet))_{i\in \mathbb{N}}$ viewed as a $\mathrm{C}(\mathbb{N}\times\mathbb{R}_+)$-valued line ensemble has the Gibbs property with respect to exponential Brownian bridge.
\end{prop}

\begin{proof}
We prove this for $(\log\mathsf{x}_i(\bullet))_{i\in \mathbb{N}}$ in that it has the Brownian Gibbs property. By applying Proposition \ref{Gibbs-finiteN} to \eqref{TransitionKernel} and then making a logarithmic change of variables or alternatively directly from \cite{JonesOConnell}, 
\begin{equation*}
(\log(N^{-1}\mathsf{x}^{(N)}(\bullet)))_{i=1}^N=\left(\xi_i^{(N)}\left(\bullet\right)-\frac{N\bullet}{2}-\log N
\right)_{i=1}^N,
\end{equation*}
for $\mathbf{x}^{(N)}(0)\in \mathbf{W}_{N;+}^{\circ}$ satisfies the Brownian Gibbs property. Moreover, from Proposition \ref{PropConvergenceOnPathSpace} and Proposition \ref{NonCollisionProp} we have convergence, as $N \to \infty$, for all $M \in \mathbb{N}$,
\begin{equation*}
(\log(N^{-1}\mathsf{x}^{(N)}(\bullet)))_{i\in \llbracket 1,M \rrbracket}\overset{\mathrm{d}}{\longrightarrow} (\log\mathsf{x}_i(\bullet))_{i\in \llbracket 1, M \rrbracket}.
\end{equation*}
 Finally, from Proposition \ref{NonCollisionProp} for any fixed time $t\ge 0$ almost surely $(\log\mathsf{x}_i(t))_{i\in \mathbb{N}}\in \mathbf{W}_\infty^\circ$. The assumption on the convergence of independent bridges is moreover trivially satisfied. Hence, we can apply Proposition \ref{PropositionLimitingGibbs} to conclude the proof. 
\end{proof}

\subsection{The evolution of the reverse characteristic polynomial}

We begin with the following result.

\begin{prop}\label{PropGLfiniteCharPolyEq}
    For any $N\in\mathbb{N}$ and $\theta\in\mathbb{R}$,
    $\mathfrak{gl}_{t;N}(z)$ satisfies the following equation
\begin{align}\label{phi^N-equ}
    \mathrm{d}\mathfrak{gl}_{t;N}(z)
    =
\mathrm{d}\mathsf{M}^N_t(z;\mathfrak{gl})
       +\left[\frac{\theta}{2}z\partial_z\mathfrak{gl}_{t;N}(z)\mathrm{d}t
    -
    \frac{z^2}{2}\partial_{z}^2\mathfrak{gl}_{t;N}(z)\right]\mathrm{d}t,
\end{align}
where $\mathsf{M}^N\in \mathrm{C}(\mathbb{R}_+;\mathsf{H}(\mathbb{C}))$ and for all $z\in \mathbb{C}$, $t\mapsto \mathsf{M}_t^N(z;\mathfrak{gl})$ are continuous local martingales with covariation, for $z,w \in \mathbb{C}$,
\begin{equation}\label{CovarGLfiniteN}
\langle \mathsf{M}_t^N(z;\mathfrak{gl}),\mathsf{M}_t^N(w;\mathfrak{gl})\rangle=\frac{zw}{z-w}\int_0^t\left(\mathfrak{gl}_{s;N}(z)\partial_w\mathfrak{gl}_{s;N}(w)-\mathfrak{gl}_{s;N}(w)\partial_z\mathfrak{gl}_{s;N}(z)\right)\mathrm{d}s.
\end{equation}

\end{prop}

\begin{proof} It will be more convenient to consider the rescaled process $\hat{\mathsf{x}}^{(N)}_i(t)=N^{-1}\mathsf{x}_i^{(N)}$ so that $\mathfrak{gl}_{t;N}(z)=\prod_{i=1}^N(1-\hat{\mathsf{x}}_i^{(N)}(t)z)$. Observe that, the $(\hat{\mathsf{x}}_i^{(N)})_{i=1}^N$ is the unique solution to the exact same equation \eqref{GLinitial} but with the scaled by $N^{-1}$ initial condition. Thus abusing notation in the computations below we still write $\mathsf{x}_i^{(N)}$ in place of $\hat{\mathsf{x}}_i^{(N)}$.

Hence, using It\^{o}'s formula applied to $\mathfrak{gl}_{t;N}(z)$, we obtain 
\begin{align*}
    \mathrm{d}\mathfrak{gl}_{t;N}(z)
    &=
    \sum_{i=1}^N\frac{-z\mathfrak{gl}_{t;N}(z)}{1-\mathsf{x}_i^{(N)}(t)z}\mathrm{d}\mathsf{x}_i^{(N)}(t)
    \\
    &=
	\mathrm{d}\mathsf{M}_t^N(z;\mathfrak{gl})
+
\sum_{i=1}^N\frac{-z\mathfrak{gl}_{t;N}(z)}{1-\mathsf{x}_i^{(N)}(t)z}
\left[
     \frac{\theta}{2}\mathsf{x}_i^{(N)}(t)
     +\sum_{j=1,j\neq i}^N\frac{\mathsf{x}_i^{(N)}(t)\mathsf{x}^{(N)}_j(t)}{\mathsf{x}^{(N)}_i(t)-\mathsf{x}^{(N)}_j(t)}
     \right]\mathrm{d}t,
\end{align*}
where the local martingale term $\mathsf{M}_t^N(z;\mathfrak{gl})$ is given by
\begin{align}\label{M_t^N}
\mathrm{d}\mathsf{M}_t^N(z;\mathfrak{gl})&=\sum_{i=1}^N\frac{-z\mathsf{x}_i^{(N)}(t)\mathfrak{gl}_{t;N}(z)}{1-\mathsf{x}_i^{(N)}(t)z}\mathrm{d}\mathsf{w}_i(t).
\end{align}
Note that, we have
\begin{align}
    \partial_z\mathfrak{gl}_{t;N}(z)
    &=
    -\sum_{i=1}^N\frac{\mathsf{x}_i^{(N)}(t)}{1-\mathsf{x}_i^{(N)}(t)z}\mathfrak{gl}_{t;N}(z),
    \label{partial_z phi}
    \\
    \partial_{z}^2\mathfrak{gl}_{t;N}(z)
    &=
    \left(\left(\sum_{i=1}^N\frac{\mathsf{x}_i^{(N)}(t)}{1-\mathsf{x}_i^{(N)}(t)z}
    \right)^2-\sum_{i=1}^N \left(\frac{\mathsf{x}_i^{(N)}(t)}{1-\mathsf{x}_i^{(N)}(t)z}\right)^2
    \right)
    \mathfrak{gl}_{t;N}(z)\\
    &=
\sum_{\underset{i\neq j}{i,j=1}}^N\frac{\mathsf{x}_i^{(N)}(t)\mathsf{x}_j^{(N)}(t)}{\left(1-\mathsf{x}_i^{(N)}(t)z\right)\left(1-\mathsf{x}_j^{(N)}(t)z\right)}
    \mathfrak{gl}_{t;N}(z).
    \label{partial_zz phi}
\end{align}
Moreover, observe that we can write,
\begin{align*}
\sum_{i=1}^N\frac{1}{1-\mathsf{x}_i^{(N)}(t)z}
    \sum_{\underset{i\neq j}{i,j=1}}^N
    \frac{\mathsf{x}_i^{(N)}(t)\mathsf{x}_j^{(N)}(t)}{\mathsf{x}^{(N)}_i(t)-\mathsf{x}^{(N)}_j(t)}
    &=
    \frac{1}{2}
    \sum_{\underset{i\neq j}{i,j=1}}^N
    \frac{\mathsf{x}_i^{(N)}(t)\mathsf{x}_j^{(N)}(t)}
    {\mathsf{x}^{(N)}_i(t)-\mathsf{x}^{(N)}_j(t)}
    \left[
    \frac{1}{1-\mathsf{x}_i^{(N)}(t)z}
    -
    \frac{1}{1-\mathsf{x}_j^{(N)}(t)z}
    \right]
    \nonumber
    \\
    &=
   \frac{z}{2}
     \sum_{\underset{i\neq j}{i,j=1}}^N\frac{\mathsf{x}_i^{(N)}(t)\mathsf{x}_j^{(N)}(t)}{\left(1-\mathsf{x}_i^{(N)}(t)z\right)\left(1-\mathsf{x}_j^{(N)}(t)z\right)}.
\end{align*}
Thus, making use of all the above, we obtain \eqref{phi^N-equ}.
Next, from 
\eqref{M_t^N},
\eqref{partial_z phi}
and
\eqref{partial_zz phi}
we can compute the quadratic covariation of $\mathsf{M}_t^N(z;\mathfrak{gl})$ and $\mathsf{M}_t^N(w;\mathfrak{gl})$, for $z,w\in \mathbb{C}$,
\begin{align*}
\frac{\mathrm{d}}{\mathrm{d}t}\langle\mathsf{M}_t^N(z;\mathfrak{gl}),
\mathsf{M}_t^N(w;\mathfrak{gl})
\rangle
&=
\sum_{i=1}^N
\frac{zw\left(\mathsf{x}_i^{(N)}(t)\right)^2}{\left(1-\mathsf{x}_i^{(N)}(t)z\right)
\left(1-\mathsf{x}_i^{(N)}(t)w\right)}
\mathfrak{gl}_{t;N}(z)
\mathfrak{gl}_{t;N}(w)
\\
&=
\frac{zw}{z-w}
\left(
\mathfrak{gl}_{t;N}(z)\partial_{w}\mathfrak{gl}_{t;N}(w)
-
\mathfrak{gl}_{t;N}(w)
\partial_z\mathfrak{gl}_{t;N}(z)\right),
\end{align*}
which completes the proof.
\end{proof}

We can now easily take the limit.

\begin{prop}\label{PropEntireEvolGL}
Let $\theta\in \mathbb{R}$ and suppose $\mathfrak{gl}_{0;N} \overset{N \to \infty}{\longrightarrow} \mathfrak{E}_{\boldsymbol{v}}^+$ in $\mathfrak{LP}_+$. Then, as $N\to \infty$,
\begin{equation*}
\mathfrak{gl}_{\bullet;N}\overset{\mathrm{d}}{\longrightarrow} \mathfrak{gl}_\bullet,
\end{equation*}
in $\mathrm{C}(\mathbb{R}_+;\mathfrak{LP}_+)$ where $\mathfrak{gl}_t(z)=\mathfrak{E}_{\mathbf{X}_{GL}^t}^+(z)$ is a Markov process in $\mathfrak{LP}_+$ which solves the equation \eqref{GLequation}.
\end{prop}

\begin{proof}
We note that we have that, see e.g. \cite{Theo-RandomEntireFunctions}, the natural maps, $\boldsymbol{v}\mapsto \mathfrak{E}_{\boldsymbol{v}}^+ \textnormal{ and } \boldsymbol{u}\mapsto\mathfrak{E}_{\boldsymbol{u}}$
are homeomorphisms from $\Upsilon_+$ to $\mathfrak{LP}_+$ and $\Upsilon$ to $\mathfrak{LP}$ respectively. In particular, this lifts to a homeomorphism from  $\mathrm{C}(\mathbb{R}_+;\Upsilon_+)$ to $\mathrm{C}(\mathbb{R}_+;\mathfrak{LP}_+)$  and from $\mathrm{C}(\mathbb{R}_+;\Upsilon)$ to $\mathrm{C}(\mathbb{R}_+;\mathfrak{LP})$. Hence, the fact that if $\mathfrak{gl}_{0;N} \overset{N \to \infty}{\longrightarrow} \mathfrak{E}_{\boldsymbol{v}}^+$ in $\mathfrak{LP}_+$, then as $N\to \infty$,
\begin{equation*}
\mathfrak{gl}_{\bullet;N}\overset{\mathrm{d}}{\longrightarrow} \mathfrak{gl}_\bullet,
\end{equation*}
in $\mathrm{C}(\mathbb{R}_+;\mathfrak{LP}_+)$ is a direct consequence of Proposition \ref{PropConvergenceOnPathSpace}. Moreover, we have that $\mathfrak{gl}_t$ is the image of $\mathbf{X}_{GL}^t$ under the map  $\boldsymbol{v}\mapsto \mathfrak{E}_{\boldsymbol{v}}^+$ and the Markov property is also immediate as this map is a measurable bijection. Now, let us couple, using Skorokhod's representation theorem, the $(\mathfrak{gl}_{\bullet;N})_{N \in \mathbb{N}}$ and $\mathfrak{gl}_{\bullet}$ on a single probability space so that convergence in $\mathrm{C}(\mathbb{R}_+;\mathfrak{LP})$, equivalently $\mathrm{C}(\mathbb{R}_+;\mathsf{H}(\mathbb{C}))$, takes place almost surely. Observe that, by Cauchy's integral formula we have for all $k\in \mathbb{N}\cup\{0\}$ that, $\partial_z^k\mathfrak{gl}_{t;N}(z)\longrightarrow\partial_z^k\mathfrak{gl}_{t}(z)$ in $\mathrm{C}(\mathbb{R}_+;\mathsf{H}(\mathbb{C}))$. Moreover, note that for $f\in \mathrm{C}(\mathbb{R}_+;\mathsf{H}(\mathbb{C}))$ the function $\mathsf{Int}_f(t)=\int_0^tf(s)\mathrm{d}s$ also belongs to $\mathrm{C}(\mathbb{R}_+;\mathsf{H}(\mathbb{C}))$ and if $f_N \longrightarrow f$ in $\mathrm{C}(\mathbb{R}_+;\mathsf{H}(\mathbb{C}))$ then also $\mathsf{Int}_{f_N} \longrightarrow \mathsf{Int}_{f}$ in $\mathrm{C}(\mathbb{R}_+;\mathsf{H}(\mathbb{C}))$. Hence, by solving for $\mathsf{M}_t^N(z;\mathfrak{gl})$ in \eqref{phi^N-equ} we obtain that it converges in $\mathrm{C}(\mathbb{R}_+;\mathsf{H}(\mathbb{C}))$ and so do the quadratic covariations from \eqref{CovarGLfiniteN} to the desired formula in \eqref{CovarEq} which completes the proof.
\end{proof}

\subsection{The Stieltjes transform and equations for the zeros}

We will derive equations for the zeros of $(\mathfrak{gl}_t)_{t\ge 0}$ or more precisely the reciprocals of them. Towards this end, define the following analogue of the Stieltjes transform in the infinite setting:
\begin{align}\label{psi}
    \psi_t^{GL}(z)
    &\overset{\textnormal{def}}{=}
\partial_z
\log\mathfrak{gl}_t(z)
=
\sum_{j=1}^\infty
\frac{-\mathsf{x}_j(t)}{1-\mathsf{x}_j(t)z}
+ 
\sum_{j=1}^\infty\mathsf{x}_j(t)
-
\boldsymbol{\gamma}(t),
\  \ z\in \mathbb{C}\setminus
\left\{\mathsf{x}_j^{-1}(t), j\in\mathbb{N}\right\}, \ \,t\ge 0.
\end{align}
We have the following result.

\begin{prop}\label{PropStieltjesGL}
The process $(\psi_t^{GL}(z))_{t\ge 0}$ satisfies the following Burgers-type equation in $\mathbb{C}\backslash \mathbb{R}$,
\begin{align}\label{psiEqu}
    \mathrm{d}\psi_t^{GL}(z)
    =&
\mathrm{d}\mathsf{m}_t\left(z\right)
    +\left[\frac{\theta}{2}\partial_z\left(z\psi_t^{GL}(z)\right)
    -
    \frac{1}{2}\partial_z
    \left(z^2\left(\psi_t^{GL}(z)\right)^2\right)\right]\mathrm{d}t,
\end{align}
where, for $z\in\mathbb{C}\backslash \mathbb{R}$, 
$t\mapsto\mathsf{m}_t(z)$ is a continuous local martingale 
with quadratic covariation, for $z,w \in \mathbb{C}\backslash \mathbb{R}$,
\begin{align}
\label{d(m(z),m(z'))}
\
    \left\langle\mathsf{m}_t(z),
    \mathsf{m}_t(w)\right\rangle
&=
\int_0^t-\partial_z\partial_w \left( zw\frac{\psi_s^{GL}(z)-\psi_s^{GL}(w)}{z-w}\right)\mathrm{d}s.
\end{align}    
\end{prop}

\begin{proof}
From Proposition \ref{PropEntireEvolGL} we can calculate for $z\in \mathbb{C}\backslash \mathbb{R}$,
\begin{align}\label{dlog phi}
    \mathrm{d}
\log\mathfrak{gl}_t(z)
&=
 \mathrm{d}\widetilde{\mathsf{M}}_t(z)
+
    \frac{\theta}{2}z\partial_z
    \log
    \mathfrak{gl}_t(z)\mathrm{d}t
  -
    \frac{z^2}{2}
    \frac{\partial_{z}^2\mathfrak{gl}_t(z)}{\mathfrak{gl}_t(z)}\mathrm{d}t
   \nonumber
   \\
  & \  \ -\frac{1}{2\mathfrak{gl}_t^2(z))}
    z^2\left(\left(\partial_z\mathfrak{gl}_t(z)\right)^2
-
\mathfrak{gl}_t(z)
\partial_{z}^2\mathfrak{gl}_t(z)
\right)\mathrm{d}t
\nonumber
\\
&=
\mathrm{d}\widetilde{\mathsf{M}}_t(z)
+
    \frac{\theta}{2}z\partial_z
    \log
    \mathfrak{gl}_t(z)\mathrm{d}t
-
    \frac{z^2}{2}\frac{\partial_{z}^2\mathfrak{gl}_t(z)}{\mathfrak{gl}_t(z)}\mathrm{d}t
   -\frac{z^2}{2}
    \left(\partial_z\log
    \mathfrak{gl}_t(z)\right)^2\mathrm{d}t
    \nonumber
    \\
    &=
    \mathrm{d}\widetilde{\mathsf{M}}_t(z)
    +\frac{\theta}{2}z\mathfrak{gl}_t(z)\mathrm{d}t
    -
    \frac{z^2}{2}\left(\mathfrak{gl}_t(z)\right)^2\mathrm{d}t,
\end{align}
where we have defined the local martingale term $\widetilde{\mathsf{M}}_t(z)$ by (note that this is well-defined as $(\mathfrak{gl}_t)_{t\ge 0}$ has only real zeros)
\begin{align*}
\widetilde{\mathsf{M}}_t(z)
&=
\int_0^t\frac{\mathrm{d}\mathsf{M}_s(z;\mathfrak{gl})}{\mathfrak{gl}_{s}(z)}.
\end{align*}
Taking the derivative in $z$, one then obtains \eqref{psiEqu} with 
\begin{align}\label{dm_t(z)}
\mathsf{m}_t\left(z\right)
&=\int_0^t\partial_z\mathrm{d}\widetilde{\mathsf{M}}_s(z).
\end{align}
The identity \eqref{d(m(z),m(z'))} for the quadratic covariation 
can then be derived, in view of \eqref{dm_t(z)}, from the corresponding expression for $\mathsf{M}_t(z;\mathfrak{gl})$.
Alternatively, one could 
take the limit $N\to\infty$ of the expressions in the finite $N$ case, by first defining $\psi_{t;N}^{GL}(z)$ as
\begin{align*}
     \psi_{t;N}^{GL}(z)
    \overset{\textnormal{def}}{=}
\partial_z
\log\mathfrak{gl}_{t;N}(z)
=
\sum_{i=1}^N
\frac{-\mathsf{x}_i^{(N)}(t)}{1-\mathsf{x}_i^{(N)}(t)z},
\end{align*}
and  noting that by virtue of $\mathfrak{gl}_{\bullet;N} \to \mathfrak{gl}_{\bullet}$ in $\mathrm{C}(\mathbb{R}_+;\mathfrak{LP}_+)$ we have that $ \psi_{\bullet;N}^{GL}\to\psi_\bullet^{GL}$ uniformly in time and compact sets in $\mathbb{C}\backslash \mathbb{R}$.
\end{proof}

We can now derive the equations, or more precisely a penultimate version in \eqref{PenultimateGLisde}, for the $\mathsf{x}_i$ via the equation for $\psi_t^{GL}$. This equation can alternatively be obtained in a different way by following the scheme of proof in Section 4 of \cite{AM24}.

\begin{prop}\label{PropISDEGL}
For any $\gamma\ge \sum_{i=1}^\infty x_i$, whenever $(x_i)_{i\in \mathbb{N}}\in \mathbf{W}_{\infty;+}^\circ$, the projection of $\mathbf{X}_{GL}^\bullet$ on the $(\mathsf{x}_i(\bullet))_{i\in \mathbb{N}}$ solves the infinite system of SDE
\begin{equation*}
\mathsf{x}_i(t)= x_i+\int_0^t\mathsf{x}_i(s)\mathrm{d}\mathsf{w}_i(s) + \frac{\theta}{2}\int_0^t\mathsf{x}_i(s)\mathrm{d}s
 +\int_0^t\sum_{j=1,j\neq i}^{\infty}\frac{\mathsf{x}_i(s)\mathsf{x}_j(s)}{\mathsf{x}_i(s)-\mathsf{x}_j(s)}\mathrm{d}s, \  \forall t\ge 0 , \   i \in \mathbb{N},
 \end{equation*}
with independent standard Brownian motions $(\mathsf{w}_i)_{i\in \mathbb{N}}$.
\end{prop}

\begin{proof}
Let us define, by virtue of Proposition \ref{NonCollisionProp}, 
\begin{align*}
\mathcal{X}_i(t)\overset{\textnormal{def}}{=}\frac{1}{\mathsf{x}_i(t)}, \  \  t\ge 0,\  \ i\in\mathbb{N}.
\end{align*}
First, observe that from \eqref{psi}  we have
\begin{align}\label{psiInvPoints-IL}
    \psi_t^{GL}(z)
=
\sum_{j=1}^\infty
\frac{1}{z-\mathcal{X}_j(t)}
+ 
\sum_{j=1}^\infty
\frac{1}{\mathcal{X}_j(t)}
-
\boldsymbol{\gamma}(t),
\end{align}
and thus, $\mathcal{X}_i(t)$'s are the poles of $\psi_t^{GL}(z)$.

Let $t\ge 0$ be arbitrary. Now, suppose
$\mathscr{C}_t^i$ is a contour enclosing 
$\mathcal{X}_i(t)$
but not 
$\mathcal{X}_j(t)$ for $i\neq j$.  We note that such a (time-evolving) contour can be taken piecewise constant in time, in the sense that it is always possible to split any interval $[0,t)$ into a finite number of sub-intervals $[t_\ell^{(i)},t_{\ell+1}^{(i)})$ so that $\mathscr{C}_t^{i}\equiv\mathscr{C}^{i;\ell}$ on $t \in [t_\ell^{(i)},t_{\ell+1}^{(i)})$. This is by virtue of the continuity and non-intersection property of the $\mathcal{X}_i$'s, see Proposition \ref{NonCollisionProp}. 

Then, by using Cauchy's residue formula, emulating the type of computations from \cite{HuangZhang}, we can write, in the differential form of the equation, 
\begin{align*}
 \mathrm{d}\mathcal{X}_i(t)&=
 \frac{1}{2\pi\textnormal{i}}
 \oint_{\mathscr{C}_t^i} 
    z\mathrm{d}\psi^{GL}_t(z)\mathrm{d}z
    \\&=
    \frac{1}{2\pi\textnormal{i}}\oint_{\mathscr{C}_t^i}
    z\mathrm{d}\mathsf{m}_t(z)\mathrm{d}z
    +
    \frac{1}{2\pi\textnormal{i}}
    \oint_{\mathscr{C}_t^i}\left(
    \frac{\theta}{2}z\partial_z\left(z\psi^{GL}_t(z)\right)
    -
   \frac{1}{2}
    z\partial_z
    \left(z^2\left(\psi^{GL}_t(z)\right)^2\right)\right)
    \mathrm{d}t\mathrm{d}z.
\end{align*}
Elementary manipulations give the following:
\begin{align*}
\frac{1}{2\pi\textnormal{i}}
    \oint_{\mathscr{C}_t^i}
z\partial_z\left(z\psi^{GL}_t(z)\right)\mathrm{d}z&
=
-\frac{1}{2\pi\textnormal{i}}
    \oint_{\mathscr{C}_t^i}
z\psi^{GL}_t(z)\mathrm{d}z
=
-\mathcal{X}_i(t),
   \\
   \frac{1}{2\pi\textnormal{i}}\oint_{\mathscr{C}_t^i} 
    z\partial_z
\left(z^2\left(\psi^{GL}_t(z)\right)^2\right)\mathrm{d}z
    &=
    -\frac{1}{2\pi\textnormal{i}}\oint_{\mathscr{C}_t^i} 
    z^2\left(\psi^{GL}_t(z)\right)^2
    \mathrm{d}z
    \\
    &\hspace{-3cm}=
    -\frac{1}{2\pi\textnormal{i}}\oint_{\mathscr{C}_t^i} 
    z^2\left[
    \frac{1}{\left(z-\mathcal{X}_i(t)\right)^2}
    +
   \frac{2}{z-\mathcal{X}_i(t)}
    \left(
    \sum_{j=1,j\neq i}^\infty
\frac{1}{z-\mathcal{X}_j(t)}
+ 
\sum_{j=1}^\infty
\frac{1}{\mathcal{X}_j(t)}
-
\boldsymbol{\gamma}(t)
    \right)
    \right]\mathrm{d}z
    \\
    &\hspace{-3cm}=
    -2\left[\mathcal{X}_i(t)
    +
    \sum_{j\neq i}\frac{\left(\mathcal{X}_i(t)\right)^2}{\mathcal{X}_i(t)-\mathcal{X}_j(t)}
    +\left(\mathcal{X}_i(t)\right)^2\left(\sum_{j=1}^\infty\frac{1}{\mathcal{X}_j(t)}
    -
    \boldsymbol{\gamma}(t)\right)\right].
\end{align*}
To characterize the local martingale term, let us define (abusing notation as these are not the same Brownian motions from before)
\begin{align}\label{tilde-w def-IL}
   \mathrm{d} \mathsf{w}_i(t)
    \overset{\textnormal{def}}{=}
    \frac{1}{2\pi\textnormal{i}\mathcal{X}_i(t)}\oint_{\mathscr{C}_t^i}
    z\mathrm{d}\mathsf{m}_t(z)\mathrm{d}z.
\end{align}
Note that,
$\left(\mathsf{w}_i(\cdot)\right)_{i\in\mathbb{N}}$
is a sequence of continuous local martingales adapted to the natural filtration of the process $(\mathbf{X}_{GL}^t)_{t\ge 0}$.
We claim, and prove shortly below, that the
$\mathsf{w}_i$'s are in fact independent 
standard Brownian motions. Then, putting everything together, it follows that, for $ i\in\mathbb{N}$,
\begin{align}\label{mathcalX-ISDE}
 \mathrm{d}\mathcal{X}_i(t)
 &=
 \mathcal{X}_i(t)\mathrm{d}\mathsf{w}_i(t)
 +
 \left(1-\frac{\theta}{2}\right)\mathcal{X}_i(t)
 \mathrm{d}t
   +
   \left[
    \sum_{j\neq i}\frac{\left(\mathcal{X}_i(t)\right)^2}{\mathcal{X}_i(t)-\mathcal{X}_j(t)}
    +\left(\mathcal{X}_i(t)\right)^2\left(\sum_{j=1}^\infty\frac{1}{\mathcal{X}_j(t)}
    -
    \boldsymbol{\gamma}(t)\right)\right]\mathrm{d}t.
\end{align}
We now prove the claim on the $\left(\mathsf{w}_i(\cdot)\right)_{i\in\mathbb{N}}$.
Using \eqref{d(m(z),m(z'))} we can compute,
\begin{align*}
    \frac{\mathrm{d}}{\mathrm{d}t}
    \langle
    \mathsf{w}_i(t),
    \mathsf{w}_j(t)
   \rangle
    &=
    \frac{-1}{(2\pi\textnormal{i})^2\mathcal{X}_i(t)
    \mathcal{X}_j(t)}\oiint_{\mathscr{C}_t^i
    \times 
    \mathscr{C}_t^j}
    zw\partial_z\partial_w \left(zw\frac{\psi^{GL}_t(z)-\psi^{GL}_t(w)}{z-w}\right)\mathrm{d}z\mathrm{d}w
    \\
     &=
    \frac{-1}{(2\pi\textnormal{i})^2\mathcal{X}_i(t)
    \mathcal{X}_j(t)}\oiint_{\mathscr{C}_t^i
    \times 
    \mathscr{C}_t^j}
    zw\left(\frac{\psi^{GL}_t(z)-\psi^{GL}_t(w)}{z-w}\right)\mathrm{d}z\mathrm{d}w
    \\
     &=
    \frac{-1}{(2\pi\textnormal{i})^2\mathcal{X}_i(t)
    \mathcal{X}_j(t)}\oiint_{\mathscr{C}_t^i
    \times 
    \mathscr{C}_t^j}
    \frac{zw}{z-w}
    \sum_{k=1}^\infty
    \left(
    \frac{1}{z-\mathcal{X}_k(t)}
    -
    \frac{1}{w-\mathcal{X}_k(t)}
    \right)
    \mathrm{d}z\mathrm{d}w.
\end{align*}
Observe now that, for $i\neq j$,
\begin{align*}
\oiint_{\mathscr{C}_t^i
    \times 
    \mathscr{C}_t^j}
    \frac{zw}{z-w}
    &\sum_{k=1}^\infty
    \left(
    \frac{1}{z-\mathcal{X}_k(t)}
    -
    \frac{1}{w-\mathcal{X}_k(t)}
    \right)
    \mathrm{d}z\mathrm{d}w
     \\&=
     -
\oiint_{\mathscr{C}_t^i
    \times 
    \mathscr{C}_t^j}
    zw
    \left(
    \frac{1}{\left(z-\mathcal{X}_i(t)\right)\left(w-\mathcal{X}_i(t)\right)}
    +
   \frac{1}{\left(z-\mathcal{X}_j(t)\right)\left(w-\mathcal{X}_j(t)\right)}
    \right)
    \mathrm{d}z\mathrm{d}w=0,
\end{align*}
while on the other hand for $j=i$
\begin{align*}
\oiint_{\mathscr{C}_t^i
    \times 
    \mathscr{C}_t^i}
    \frac{zw}{z-w}
    \sum_{k=1}^\infty
    \left(
    \frac{1}{z-\mathcal{X}_k(t)}
    -
    \frac{1}{w-\mathcal{X}_k(t)}
    \right)
    \mathrm{d}z\mathrm{d}w
     &=
     -
\oiint_{\mathscr{C}_t^i
    \times 
    \mathscr{C}_t^i}
    \frac{zw}{\left(z-\mathcal{X}_i(t)\right)\left(w-\mathcal{X}_i(t)\right)}\mathrm{d}z\mathrm{d}w
    \\
    &=-(2\pi\textnormal{i})^2
    \mathcal{X}_i(t)^2.
\end{align*}
Hence, we obtain, for all $t\ge 0$, $i,j \in \mathbb{N}$,
\begin{align*}
    \langle
   \mathsf{w}_i(t),
    \mathsf{w}_j(t)
    \rangle=\mathbf{1}_{i=j}t,
\end{align*}
which concludes the proof of the claim by Levy's chararacterisation of Brownian motion \cite{Kallenberg}. Finally,  we can apply It\^{o}'s  formula to \eqref{mathcalX-ISDE} to obtain 
\begin{align}\label{PenultimateGLisde}
     \mathrm{d}\mathsf{x}_i(t)= \mathsf{x}_i(t)\mathrm{d}\mathsf{w}_i(t) + \frac{\theta}{2}\mathsf{x}_i(t)\mathrm{d}t
 +
 \left(\boldsymbol{\gamma}(t)-\sum_{j=1}^\infty\mathsf{x}_j(t) \right)\mathrm{d}t
 +\sum_{j=1,j\neq i}^{\infty}\frac{\mathsf{x}_i(t)\mathsf{x}_j(t)}{\mathsf{x}_i(t)-\mathsf{x}_j(t)}\mathrm{d}t, \  \ i\in\mathbb{N}.
\end{align}
The equation above is still different from the desired equation 
    by the additional drift term
    $ \boldsymbol{\gamma}(t)-\sum_{j=1}^{\infty}\mathsf{x}_j(t)$. However, one can follow the exact same argument given in Section 4 of \cite{AM24}  to show that, almost surely,
 \begin{align*}
\boldsymbol{\gamma}(t)=\sum_{i=1}^{\infty}\mathsf{x}_i(t), \ \  \forall t>0,
 \end{align*}
 yielding the desired SDE \eqref{ISDEGL} for the $\mathsf{x}_i$'s.
\end{proof}

\subsection{Proof of Theorem \ref{ThmMainGLN}}

We now complete the proof of Theorem \ref{ThmMainGLN}.

\begin{proof}[Proof of Theorem \ref{ThmMainGLN}]
All statements of the theorem have already been proven in Propositions \ref{PropConvergenceOnPathSpace}, \ref{PropGibbsPropertyGL} and \ref{PropISDEGL} except the final one. For this, first observe that for functions in $\mathfrak{LP}_+$ we have:
\begin{equation*}
\left|\mathfrak{E}_{\boldsymbol{v}}^+(z)\right|\le \mathrm{e}^{3\gamma|z|}, \ \ \forall z \in \mathbb{C}.
\end{equation*}
Hence, it suffices to show $\boldsymbol{\gamma}(t) \overset{t\longrightarrow \infty}{\longrightarrow} 0$ almost surely. Observe that, $\boldsymbol{\gamma}(t)=-\partial_z\mathfrak{gl}_t(z)|_{z=0}$. Then, using the equation for $(\mathfrak{gl}_t)_{t\ge0}$ we obtain,
\begin{equation*}
\mathrm{d}\boldsymbol{\gamma}(t) =\mathrm{d}\mathcal{M}(t)+\frac{\theta}{2}\boldsymbol{\gamma}(t)\mathrm{d}t
\end{equation*}
where $\mathcal{M}(t)$ is a continuous local martingale (whose covariation we do not need). Hence, using an integrating factor we can integrate this equation to obtain,
\begin{equation*}
\mathrm{e}^{-t\theta/2}\boldsymbol{\gamma}(t)=\boldsymbol{\gamma}(0)+\int_0^t\mathrm{e}^{s\theta/2}\mathrm{d}\mathcal{M}(s), \ \forall t\ge 0.
\end{equation*}
Now, the right hand side is a non-negative (since $\boldsymbol{\gamma}$ is non-negative) continuous local martingale and thus a non-negative supermartingale. By the supermartingale convergence theorem it converges almost surely to a finite limit. Hence, for $\theta<0$, as $t \longrightarrow \infty$, $\boldsymbol{\gamma}(t) \longrightarrow 0$ almost surely which completes the proof.
\end{proof}

\section{Dynamical models for stochastic zeta functions}

\subsection{Proof of results on the Rider-Valk\'{o} model}

\begin{proof}[Proof of statements on $\mathfrak{rv}_t$ from Theorem \ref{ThmMainStat}]  In \cite{AM24} we proved that there exists a Feller diffusion process $(\mathbf{X}_{RV}^t)_{t\ge 0}$ on $\Upsilon_+$ such that whenever, 
\begin{equation*}
\left((N^{-1}\mathsf{z}_i^{(N)}(0))_{i\in \mathbb{N}},N^{-1}\sum_{i=1}^\infty \mathsf{z}_i^{(N)}(0)\right)  \overset{N \to \infty}{\longrightarrow} \boldsymbol{v}=((x_i)_{i\in \mathbb{N}},\gamma)\in \Upsilon_+,
\end{equation*}
we have, as $N \to \infty$,
\begin{equation*}
\left((N^{-1}\mathsf{z}_i^{(N)}(\bullet))_{i\in \mathbb{N}},N^{-1}\sum_{i=1}^\infty \mathsf{z}_i^{(N)}(\bullet)\right) \overset{\mathrm{d}}{\longrightarrow} \left((\mathsf{z}_i(\bullet))_{i\in \mathbb{N}},\boldsymbol{\gamma}_{RV}(\bullet)\right)\overset{\mathrm{def}}{=} \mathbf{X}_{RV}^\bullet
\end{equation*}
in $\mathrm{C}(\mathbb{R}_+;\Upsilon_+)$
where $\mathbf{X}_{RV}^0=\boldsymbol{v}$. From this, the fact that if $\mathfrak{rv}_{0;N} \overset{N \to \infty}{\longrightarrow} \mathfrak{E}_{\boldsymbol{v}}^+$ in $\mathfrak{LP}_+$, then
\begin{equation*}
\mathfrak{rv}_{\bullet;N}\overset{\mathrm{d}}{\longrightarrow} \mathfrak{rv}_\bullet,
\end{equation*}
in $\mathrm{C}(\mathbb{R}_+;\mathfrak{LP}_+)$ and  $(\mathfrak{rv}_t)_{t\ge 0}$ is a Markov process follows immediately under the homeomorphism between $\mathrm{C}(\mathbb{R}_+;\Upsilon_+)$ and $\mathrm{C}(\mathbb{R}_+;\mathfrak{LP}_+)$. The proof of the dynamics \eqref{RVspde} of $\mathfrak{rv}_t$ follows word-for-word the computations for $\mathfrak{gl}_t$. The only additional drift term for $\mathfrak{rv}_{t;N}$ comes from the constant term in the SDE \eqref{RVmodel} which becomes
\begin{equation*}
-\frac{z}{2N}\sum_{i=1}^N\frac{\mathfrak{rv}_{t;N}\left(z\right)}{1-\frac{\mathsf{z}_i^{(N)}(t)z}{N}}
=-\frac{z}{2}\mathfrak{rv}_{t;N}(z)
    +
\frac{z^2}{2N}\partial_z\mathfrak{rv}_{t;N}(z), 
\end{equation*}
and in the limit gives the additional drift $-z\mathfrak{rv}_{t}(z)/2\times\mathrm{d}t$ in \eqref{RVspde}.
Regarding convergence to equilibrium, we have from \cite{AM24} that, for $\nu>-1$, as $t\to \infty$, 
\begin{equation*}
\mathbf{X}_{RV}^t \overset{\mathrm{d}}{\longrightarrow} \mathbf{Z}_{Bes;\nu}
\end{equation*}
where the law of $\mathbf{Z}_{Bes;\nu}$ on $\Upsilon_+$ is concentrated on $\Upsilon_+^\star$, given by 
\begin{equation*}
\Upsilon_+^\star=\left\{\boldsymbol{v}=(\mathbf{x},\gamma)\in \Upsilon_+: \gamma=\sum_{i=1}^\infty x_i\right\}
\end{equation*}
and the law of the $\{x_i\}$ is given by $\mathcal{IB}_\nu$ (note that this completely characterises the distribution of $\mathbf{Z}_{Bes;\nu}$).
Since the map $\boldsymbol{v}\mapsto \mathfrak{E}_{\boldsymbol{v}}^+$ is continuous the convergence statement \eqref{ConvToEqIB} follows from the continuous mapping theorem and observing that the pushforward of the law of $\mathbf{Z}_{Bes;\nu}$ under this map is exactly the law of $\mathbb{B}_{\nu}$.
\end{proof}

\subsection{The limiting entire function of the dynamical Cauchy model}

We begin with the following result on the  dynamics of the finite-$N$ reverse characteristic polynomial.

\begin{prop}\label{PropHPcharpolySDE}
    For any $N\in\mathbb{N}$ and $\mathfrak{s}\in\mathbb{C}$,
    $\mathfrak{hp}_{t;N}(z)$ satisfies the following equation
\begin{align}
\mathrm{d}\mathfrak{hp}_{t;N}(z)
    =
\mathrm{d}\mathsf{M}^N_t(z;\mathfrak{hp})
       +\mathfrak{B}^{(N)}_z\mathfrak{hp}_{t;N}(z)\mathrm{d}t,
\end{align}
with the linear operator $\mathfrak{B}^{(N)}_z$ acting on $f\in \mathsf{H}(\mathbb{C})$ as follows
\begin{align}
\mathfrak{B}^{(N)}_zf(z)&=-z^2\partial_z^2f(z)-2z\Re(\mathfrak{s})\partial_zf(z)-2z\Im(\mathfrak{s})f(z)-z^2f(z)\nonumber\\&+\frac{1}{N}\left[-z^2f(z)+\left(2z^2\Im(\mathfrak{s})+2z^3-\frac{2z^3}{N}\right)\partial_zf(z)-\frac{1}{N}\partial_z^2f(z)\right].
\end{align}
and moreover where $\mathsf{M}^N\in \mathrm{C}(\mathbb{R}_+;\mathsf{H}(\mathbb{C}))$ and for all $z\in \mathbb{C}$, $t\mapsto \mathsf{M}_t^N(z;\mathfrak{hp})$ are continuous local martingales with covariation, for $z,w \in \mathbb{C}$,
\begin{align}\label{HPQuadraticCovarFiniteN}
\langle \mathsf{M}_t^N(z;\mathfrak{hp}),\mathsf{M}_t^N(w;\mathfrak{hp})\rangle=\frac{2zw}{z-w}\int_0^t\left(\mathfrak{hp}_{s;N}(z)\partial_w\mathfrak{hp}_{s;N}(w)-\mathfrak{hp}_{s;N}(w)\partial_z\mathfrak{hp}_{s;N}(z)+\mathsf{E}_{s;N}(z,w)\right)\mathrm{d}s,
\end{align}
and the term $\mathsf{E}_{s;N}(z,w)$, with $s\ge 0$, $z,w\in \mathbb{C}$, is given by,
\begin{equation}\label{ErrorCovariance}
\mathsf{E}_{s;N}(z,w)=\frac{2zw}{N^2}\sum_{i=1}^N \frac{\mathfrak{hp}_{s;N}(z)\mathfrak{hp}_{s;N}(w)}{\left(1-\frac{\mathsf{y}_i^{(N)}(s)z}{N}\right)\left(1-\frac{\mathsf{y}_i^{(N)}(s)w}{N}\right)}.
\end{equation}
\end{prop}

\begin{proof}
Again it is more convenient to consider the rescaled coordinates $\hat{\mathsf{y}}_i^{(N)}(t)=N^{-1}\mathsf{y}_i^{(N)}(t)$ so that $\mathfrak{hp}_{t;N}(z)=\prod_{i=1}^N(1-\hat{\mathsf{y}}_i^{(N)}(t)z)$. Observe that, the $(\hat{\mathsf{y}}_i^{(N)})_{i=1}^N$ solve the  following SDE,
\begin{align}\label{Hua-Pickrell}
\mathrm{d}\hat{\mathsf{y}}_i^{(N)}(t)&=\sqrt{2\left((\hat{\mathsf{y}}_i^{(N)}(t))^2+N^{-2}\right)}\mathrm{d}\mathsf{w}_i(t)+2\left[-\Re{(\mathfrak{s})}\hat{\mathsf{y}}_i^{(N)}(t)+\frac{\Im(\mathfrak{s})}{N}\right]\mathrm{d}t\nonumber\\
&+2\sum_{j=1,j\neq i}^N \frac{\hat{\mathsf{y}}_i^{(N)}(t)\hat{\mathsf{y}}_j^{(N)}(t)+N^{-2}}{\hat{\mathsf{y}}_i^{(N)}(t)-\hat{\mathsf{y}}_j^{(N)}(t)}\mathrm{d}t, \ \ i\in \llbracket 1, N \rrbracket.
\end{align}
Hence, using It\^{o}'s formula applied to $\mathfrak{hp}_{t;N}(z)$, we obtain 
\begin{align*}
    \mathrm{d}\mathfrak{hp}_{t;N}(z)
    &=
    \sum_{i=1}^N\frac{-z\mathfrak{hp}_{t;N}(z)}{1-\hat{\mathsf{y}}_i^{(N)}(t)z}\mathrm{d}\hat{\mathsf{y}}_i^{(N)}(t)
    \\
    &=
	\mathrm{d}\mathsf{M}_t^N(z;\mathfrak{hp})
-\sum_{i=1}^N\frac{2z\mathfrak{hp}_{t;N}(z)}{1-\hat{\mathsf{y}}_i^{(N)}(t)z}
\left[-\Re{(\mathfrak{s})}\hat{\mathsf{y}}_i^{(N)}(t)+\frac{\Im{(\mathfrak{s})}}{N}+\sum_{j\neq i}\frac{\hat{\mathsf{y}}_i^{(N)}(t)\hat{\mathsf{y}}_j^{(N)}(t)+N^{-2}}{\hat{\mathsf{y}}_i^{(N)}(t)-\hat{\mathsf{y}}_j^{(N)}(t)}\right]\mathrm{d}t,
\end{align*}
where the local martingale term $\mathsf{M}_t^{N}(z;\mathfrak{hp})$ is given by,
\begin{equation*}
\mathsf{M}_t^{N}(z;\mathfrak{hp})=
\sum_{i=1}^N\frac{-z\sqrt{2\left(\hat{\mathsf{y}}_i^{(N)}(t)^2+N^{-2}\right)}\mathfrak{hp}_{t;N}(z)}{1-\hat{\mathsf{y}}_i^{(N)}(t)z}\mathrm{d}\mathsf{w}_i(t).
\end{equation*}
Then, using  the identities from the proof of Proposition \ref{PropGLfiniteCharPolyEq} along with the identity,
\begin{align*}
&2z\mathfrak{hp}_{t;N}(z)\sum_{i=1}^N\frac{1}{1-z\hat{\mathsf{y}}_i^{(N)}(t)}\sum_{j\neq i}\frac{1}{\hat{\mathsf{y}}_i^{(N)}(t)-\hat{\mathsf{y}}_j^{(N)}(t)}\\&=z^2\left[N(N-1)\mathfrak{hp}_{t;N}(z)-2z(N-1)\partial_z \mathfrak{hp}_{t;N}(z)+z^2\partial_z^2\mathfrak{hp}_{t;N}(z)\right],
\end{align*}
and elementary manipulations we can readily simplify the equation for $(\mathfrak{hp}_{t;N})_{t\ge 0}$ to the desired expression in the statement of the proposition.
\end{proof}

\begin{proof}[Proof of statements on $\mathfrak{hp}_t$ from Theorem \ref{ThmMainStat}] Write $[x]_+=\max\{0,x\}$. In \cite{AM24} we proved that there exists a Feller diffusion process $(\mathbf{X}_{HP}^t)_{t\ge 0}$ on $\Upsilon$ such that whenever, 
\begin{align*}
\bigg((N^{-1}[\mathsf{y}_i^{(N)}(0)]_+)_{i\in \mathbb{N}},(N^{-1}[-\mathsf{y}_{N-i+1}^{(N)}(0)]_+)_{i\in \mathbb{N}},N^{-1}\sum_{i=1}^\infty \mathsf{y}_i^{(N)}(0),N^{-2}\sum_{i=1}^\infty \left(\mathsf{y}_i^{(N)}(0)\right)^2\bigg)  \overset{N \to \infty}{\longrightarrow} \boldsymbol{v}\in \Upsilon,
\end{align*}
we have, as $N \to \infty$,
\begin{align}
\bigg((N^{-1}[\mathsf{y}_i^{(N)}(\bullet)]_+)_{i\in \mathbb{N}},(N^{-1}[-\mathsf{y}_{N-i+1}^{(N)}(\bullet)]_+)_{i\in \mathbb{N}},N^{-1}\sum_{i=1}^\infty \mathsf{y}_i^{(N)}(\bullet),N^{-2}\sum_{i=1}^\infty \left(\mathsf{y}_i^{(N)}(\bullet)\right)^2\bigg) \nonumber\\\overset{\mathrm{d}}{\longrightarrow} \left((\mathsf{y}_i(\bullet))_{i\in \mathbb{N}},(\mathsf{y}_{-i}(\bullet))_{i\in \mathbb{N}},\boldsymbol{\gamma}_{HP}(\bullet),\boldsymbol{\delta}(\bullet)\right)\overset{\mathrm{def}}{=} \mathbf{X}_{HP}^\bullet,\label{HPFellerProcessDef}
\end{align}
in $\mathrm{C}(\mathbb{R}_+;\Upsilon)$
where $\mathbf{X}_{HP}^0=\boldsymbol{v}$. From this, the fact that if $\mathfrak{hp}_{0;N} \overset{N \to \infty}{\longrightarrow} \mathfrak{E}_{\boldsymbol{v}}$ in $\mathfrak{LP}$, then
\begin{equation*}
\mathfrak{hp}_{\bullet;N}\overset{\mathrm{d}}{\longrightarrow} \mathfrak{hp}_\bullet,
\end{equation*}
in $\mathrm{C}(\mathbb{R}_+;\mathfrak{LP})$ and  $(\mathfrak{hp}_t)_{t\ge 0}$ is a Markov process follows immediately under the homeomorphism between $\mathrm{C}(\mathbb{R}_+;\Upsilon)$ and $\mathrm{C}(\mathbb{R}_+;\mathfrak{LP})$. To get the limiting dynamics \eqref{HPspde} for $(\mathfrak{hp}_t)_{t\ge 0}$ we follow the same scheme as in the proof of Proposition \ref{PropEntireEvolGL} by taking the limit of the equation for $(\mathfrak{hp}_{t;N})_{t\ge 0}$ from Proposition \ref{PropHPcharpolySDE}. It is easy to see that the drift term converges uniformly to the desired limiting drift. The only thing that remains is to check the martingale terms quadratic covariations do so as well. Towards this end, recall that for functions in $\mathfrak{LP}$ we have the bound,
for some $C<\infty$,
\begin{equation*}
|\mathfrak{E}_{\boldsymbol{v}}(z)|\le \mathrm{e}^{C(|\gamma||z|+\delta|z|^2)}, \ \forall z\in \mathbb{C}.
\end{equation*}
Then, for any $T\ge 0$ and compact set $\mathscr{K}\subset \mathbb{C}$ we can bound the error $\mathsf{E}_{t;N}(z,w)$ as follows, by applying the above inequality to each term in the sum from \eqref{ErrorCovariance} defining it, with a constant $C_{\mathscr{K}}$ only depending on $\mathscr{K}$,
\begin{align*}
\sup_{t\in [0,T]}\sup_{z,w\in \mathscr{K}} \left|\mathsf{E}_{t;N}(z,w)\right|&\le \sup_{t \in [0,T]} \frac{C_{\mathscr{K}}}{N^2}\sum_{i=1}^N\mathrm{e}^{C_{\mathscr{K}}\left(\left|N^{-1}\sum_{i=1}^N\mathsf{y}_i^{(N)}(t)\right|+N^{-1}\max\left\{|\mathsf{y}_1^{(N)}(t)|,\mathsf{y}_N^{(N)}(t)\right\}+N^{-2}\sum_{i=1}^N \left(\mathsf{y}_i^{(N)}(t)\right)^2\right)}\\&=\frac{C_{\mathscr{K}}}{N}\mathrm{e}^{C_{\mathscr{K}}\sup_{t\in [0,T]}\left(\left|N^{-1}\sum_{i=1}^N\mathsf{y}_i^{(N)}(t)\right|+N^{-1}\max\left\{|\mathsf{y}_1^{(N)}(t)|,\mathsf{y}_N^{(N)}(t)\right\}+N^{-2}\sum_{i=1}^N \left(\mathsf{y}_i^{(N)}(t)\right)^2\right)},
\end{align*}
which, by virtue of \eqref{HPFellerProcessDef} as $N \to \infty$, goes to $0$ almost surely in the coupling afforded by the Skorokhod representation theorem. Finally, regarding convergence to equilibrium, we have from \cite{AM24} that, for $\mathfrak{s}>-1/2$ as $t\to \infty$, 
\begin{equation*}
\mathbf{X}_{HP}^t \overset{\mathrm{d}}{\longrightarrow} \mathbf{Z}_{HP;\mathfrak{s}}
\end{equation*}
where the law of $\mathbf{Z}_{HP;\mathfrak{s}}$ on $\Upsilon$ is concentrated on $\Upsilon^\star$, given by 
\begin{align*}
 \Upsilon^\star&=\left\{\boldsymbol{u}=(\mathbf{x}^+,\mathbf{x}^-,\gamma,\delta)\in \Upsilon: \gamma= \lim_{k\to \infty}\left[\sum_{i=1}^\infty x_i^+\mathbf{1}_{x_i^+>k^{-2}}-\sum_{i=1}^\infty x_i^-\mathbf{1}_{x_i^->k^{-2}}\right], \delta=\sum_{i=1}^\infty (x_i^+)^2+(x_i^-)^2\right\}, 
\end{align*}
and the law of the $\{x_i^+,-x_i^-\}$ is given by $\mathcal{HP}_\mathfrak{s}$ (note that this completely characterises the distribution of $\mathbf{Z}_{HP;\mathfrak{s}}$).
Since the map $\boldsymbol{v}\mapsto \mathfrak{E}_{\boldsymbol{v}}$ is continuous the convergence statement \eqref{ConvToEqIB} follows from the continuous mapping theorem and observing that the pushforward of the law of $\mathbf{Z}_{HP;\mathfrak{s}}$ under this map is exactly the law of $\mathbb{HP}_{\mathfrak{s}}$.
\end{proof}

\subsection{Proof of Proposition \ref{HPISDE}}

We now prove Proposition \ref{HPISDE} under Assumption \ref{ConjHP}.

\begin{proof}[Proof of Proposition \ref{HPISDE}]

For any $t\ge 0$, we define 
\begin{align}\label{psi-HP}
    \psi^{HP}_t(z)
    \overset{\textnormal{def}}{=}&
\partial_z
\log\mathfrak{hp}_t(z)
\nonumber=
\sum_{i\in \mathbb{Z}\backslash \{0\}}
\left(
\mathsf{y}_i(t)+
\frac{-\mathsf{y}_i(t)}{1-\mathsf{y}_i(t)z}
\right)
-
\boldsymbol{\gamma}_{HP}(t)-
z\left(\boldsymbol{\delta}(t)
-\sum_{i\in \mathbb{Z}\backslash \{0\}}
\left(\mathsf{y}_i(t)\right)^2\right),\\
&z\in\mathbb{C}\setminus
\left\{\left(\mathsf{y}_j(t)\right)^{-1}, j\in\mathbb{Z}\backslash\{0\}\right\},
\end{align}
where recall the $\mathsf{y}_i(t)$'s are the reciprocals of the ordered zeros of $\mathfrak{hp}_t$ as in Assumption \ref{Assumption} and \eqref{RecipZeros}. Note that, this is consistent with \eqref{HPFellerProcessDef}. By performing analogous computations as
in Proposition \ref{PropStieltjesGL}, we obtain using 
equation \eqref{HPspde} that,
\begin{align}\label{psiEqu-HP}
    \mathrm{d}\psi^{HP}_t(z)
    =&
    \mathrm{d}\mathsf{u}_t(z)
    -
   2\Re{(\mathfrak{s})}\partial_z\left(z\psi^{HP}_t(z)\right)\mathrm{d}t
    -
    2\left(\Im{(\mathfrak{s})}+z\right)
    \mathrm{d}t
    -
    \partial_z
    \left(z^2\left(\psi^{HP}_t(z)\right)^2\right)\mathrm{d}t,
\end{align}
with the martingale term $\mathsf{u}_t(z)$ having quadratic covariation,
\begin{align*}
\frac{\mathrm{d}}{\mathrm{d}t}\left\langle\mathsf{u}_t(z),
\mathsf{u}_t(w)
\right\rangle
&=
-2\partial_z\partial_w\left( zw\frac{\psi^{HP}_t(z)-\psi^{HP}_t(w)}{z-w}\right).
\end{align*}
For any $i\in\mathbb{Z}\setminus\{0\}, t\ge 0$, recalling 
\eqref{RecipZeros}, the display \eqref{psi-HP} is then read as
\begin{align}\label{psiInvPoints-HP}
    \psi^{HP}_t(z)
=&
\sum_{i\in\mathbb{Z}\setminus\{0\}}
\left(
\frac{1}{\mathfrak{y}_i(t)}+
\frac{1}{z-\mathfrak{y}_i(t)}
\right)-
\boldsymbol{\gamma}_{HP}(t)
-
z\left(\boldsymbol{\delta}(t)
-\sum_{i\in\mathbb{Z}\setminus\{0\}}
\left(\frac{1}{\mathfrak{y}_i(t)}\right)^2\right)
\nonumber\\
=&
\sum_{i\in\mathbb{Z}\setminus\{0\}}
\left(
\frac{z}{\mathfrak{y}_i(t)\left(z-\mathfrak{y}_i(t)\right)}
\right)-
\boldsymbol{\gamma}_{HP}(t)
-
z\left(\boldsymbol{\delta}(t)
-\sum_{i\in\mathbb{Z}\setminus\{0\}}
\left(\frac{1}{\mathfrak{y}_i(t)}\right)^2\right).
\end{align}
As in the proof of Proposition \ref{PropISDEGL}, let $t\ge 0$ be arbitrary and suppose
$\mathscr{C}_t^i$ is a contour enclosing 
$\mathfrak{y}_i(t)$
but not 
$\mathfrak{y}_j(t)$ for $i\neq j$. Then, using the residue theorem we can compute,
\begin{align*}
 \mathrm{d}\mathfrak{y}_i(t)&=
 \frac{1}{2\pi\textnormal{i}}
 \oint_{\mathscr{C}_t^i} 
    z\mathrm{d}\psi^{HP}_t(z)\mathrm{d}z=
    \frac{1}{2\pi\textnormal{i}}\oint_{\mathscr{C}_t^i}
    z\mathrm{d}\mathsf{u}_t(z)\mathrm{d}z
    \\&-
    \frac{1}{2\pi\textnormal{i}}
    \oint_{\mathscr{C}_t^i}
    \left(
2\Re{(\mathfrak{s})}\partial_z\left(z\psi^{HP}_t(z)\right)
    +
    2\left(\Im{(\mathfrak{s})}+z\right)z
    +
    \partial_z
\left(z^2\left(\psi^{HP}_t(z)\right)^2\right)\right)
    \mathrm{d}t\mathrm{d}z.
\end{align*}
Observe that, by analyticity
\begin{align*}
\oint_{\mathscr{C}_t^i} \left(\Im{(\mathfrak{s})}+z\right)z
    \mathrm{d}z
    =0.
    \end{align*}
We thus compute the other terms in the expression as follows
\begin{align*}
&\frac{1}{2\pi\textnormal{i}}
    \oint_{\mathscr{C}_t^i}
z\partial_z\left(z\psi^{HP}_t(z)\right)\mathrm{d}z
=
-\frac{1}{2\pi\textnormal{i}}
    \oint_{\mathscr{C}_t^i}
z\psi^{HP}_t(z)\mathrm{d}z
=
-\mathfrak{y}_i(t),\\
   &\frac{1}{2\pi\textnormal{i}}\oint_{\mathscr{C}_t^i} 
    z\partial_z
\left(z^2\left(\psi^{HP}_t(z)\right)^2\right)\mathrm{d}z
    =
    -\frac{1}{2\pi\textnormal{i}}\oint_{\mathscr{C}_t^i}
    z^2\left(\psi^{HP}_t(z)\right)^2
    \mathrm{d}z=
    -\frac{1}{2\pi\textnormal{i}}\oint_{\mathscr{C}_t^i} 
    \left(
\frac{z^2}{\mathfrak{y}_i(t)\left(z-\mathfrak{y}_i(t)\right)}
\right)^2\mathrm{d}z
\\
&-\frac{1}{2\pi\textnormal{i}}\oint_{\mathscr{C}_t^i} 
\frac{2z^3}{\mathfrak{y}_i(t)\left(z-\mathfrak{y}_i(t)\right)}
\left[
\sum_{j\in\mathbb{Z}\setminus\{0,i\}}
\left(
\frac{z}{\mathfrak{y}_j(t)\left(z-\mathfrak{y}_j(t)\right)}
\right)-
\boldsymbol{\gamma}_{HP}(t)
-
z\left(\boldsymbol{\delta}(t)
-\sum_{j\in\mathbb{Z}\setminus\{0\}}
\left(\frac{1}{\mathfrak{y}_j(t)}\right)^2\right)
    \right]\mathrm{d}z\\
    &=
    -4\mathfrak{y}_i(t)
    -2\left[
    \sum_{j\in\mathbb{Z}\setminus\{0,i\}}
\left(
\frac{\left(\mathfrak{y}_i(t)\right)^3}{\mathfrak{y}_j(t)\left(\mathfrak{y}_i(t)-\mathfrak{y}_j(t)\right)}
\right)-
\boldsymbol{\gamma}_{HP}(t)\mathfrak{y}_i(t)^2
-
\left(\boldsymbol{\delta}(t)
-\sum_{j\in\mathbb{Z}\setminus\{0\}}
\left(\frac{1}{\mathfrak{y}_j(t)}\right)^2\right)
\left(\mathfrak{y}_i(t)\right)^3
\right].
\end{align*}
As for the diffusion term, one can argue similarly to Proposition \ref{PropISDEGL} to verify that 
\begin{align*}
    \frac{1}{2\pi\textnormal{i}}\oint_{\mathscr{C}_t^i}
    z\mathrm{d}\mathsf{u}_t(z)\mathrm{d}z= 
    \sqrt{2}\mathfrak{y}_i(t)
    \mathrm{d}\mathsf{w}_i(t),
\end{align*}
with
$(\mathsf{w}_i(\cdot))_{i\in\mathbb{Z\setminus\{0\}}}$
being a sequence of independent standard Brownian motions
adapted to the natural filtration of
$(\mathbf{X}^{t}_{HP})_{t\ge 0}$. Putting everything together we obtain 
 \begin{align*}
\mathrm{d}\mathfrak{y}_i(t)
    =& 
    \sqrt{2}\mathfrak{y}_i(t)
    \mathrm{d}\mathsf{w}_i(t)
    +2\left(\Re{(\mathfrak{s})}
    +2\right)\mathfrak{y}_i(t)\mathrm{d}t
    +
    2
    \sum_{j\in\mathbb{Z}\setminus\{0,i\}}
\left(
\frac{\left(\mathfrak{y}_i(t)\right)^3}{\mathfrak{y}_j(t)\left(\mathfrak{y}_i(t)-\mathfrak{y}_j(t)\right)}
\right)\mathrm{d}t
\\
&
-2\left[
\boldsymbol{\gamma}_{HP}(t)\left(\mathfrak{y}_i(t)\right)^2
+
\left(\boldsymbol{\delta}(t)
-\sum_{j\in\mathbb{Z}\setminus\{0\}}
\left(\frac{1}{\mathfrak{y}_j(t)}\right)^2\right)
\left(\mathfrak{y}_i(t)\right)^3
\right]\mathrm{d}t,\  \ i\in\mathbb{Z}\setminus\{0\}.
\end{align*}
Then,
an application of It\^{o}'s  formula yields the following equation for the $\mathsf{y}_i$,
\begin{align}\label{initial ISDE-HP}
  \mathrm{d}\mathsf{y}_i(t)
  =& 
 \sqrt{2}\mathsf{y}_i(t)
 \mathrm{d}\mathsf{w}_i(t)
-2\left(\Re{(\mathfrak{s})}
+1\right)\mathsf{y}_i(t)\mathrm{d}t
  +
 2
 \sum_{j\in\mathbb{Z}\setminus\{0,i\}}
\frac{\mathsf{y}_j^2(t)}{\mathsf{y}_i(t)-\mathsf{y}_j(t)}
\mathrm{d}t
\nonumber\\
&
+2\left[
\boldsymbol{\gamma}_{HP}(t)
+
\frac{1}{\mathsf{y}_i(t)}\left(\boldsymbol{\delta}(t)
-\sum_{j\in\mathbb{Z}\setminus\{0\}}
\mathsf{y}_j^2(t)\right)
\right]\mathrm{d}t,\  \ i\in\mathbb{Z}\setminus\{0\}.
\end{align}
We next obtain the SDE satisfied by $(\boldsymbol{\gamma}_{HP}(t))_{t\ge0}$. It is easy to check from the definition that $\boldsymbol{\gamma}_{HP}(t)=-\partial_z\mathfrak{hp}_t(z)\big|_{z=0}$. Therefore, we can use
\eqref{HPspde} to obtain the equation
\begin{align}\label{gammaSDE-HP}
\mathrm{d}\boldsymbol{\gamma}_{HP}(t)
=&
-\partial_z \mathrm{d}\mathsf{M}_t(z;\mathfrak{hp})\big|_{z=0}
+
2\left(\Im{(\mathfrak{s})}-\Re{(\mathfrak{s})}\boldsymbol{\gamma}_{HP}(t)\right)\mathrm{d}t.
\end{align}
Moving on, we have using It\^{o}'s formula that 
\begin{align}\label{x^2-ISDE-HP}
		\frac{1}{4}\mathrm{d}\mathsf{y}_i^2(t)  
        =& 
 \frac{\sqrt{2}}{2}\mathsf{y}_i^2(t)
 \mathrm{d}\mathsf{w}_i(t)
-\left(\Re{(\mathfrak{s})}
+\frac{1}{2}\right)\mathsf{y}_i^2(t)\mathrm{d}t
  +
 \sum_{j\in\mathbb{Z}\setminus\{0,i\}}
\frac{\mathsf{y}_i(t)\mathsf{y}_j^2(t)}{\mathsf{y}_i(t)-\mathsf{y}_j(t)}
\mathrm{d}t
+
\boldsymbol{\gamma}_{HP}(t)\mathsf{y}_i(t)
\mathrm{d}t
\nonumber\\
&+\left(\boldsymbol{\delta}(t)
-\sum_{j\in\mathbb{Z}\setminus\{0\}}
\mathsf{y}_j^2(t)\right)
\mathrm{d}t,\  \ i\in\mathbb{Z}\setminus\{0\}.
\end{align}
Writing the equations in the integral form, and taking the sum over the first $K\in\mathbb{N}$ coordinates, we obtain, for any $t>0$,
\begin{align}\label{x^2-intISDE-HP}
		\frac{1}{4}\sum_{i=1}^K\mathsf{y}_i^2(t)  
        =& 
        \frac{1}{4}\sum_{i=1}^Ky_i^2
        +
\mathsf{MG}_K(t)
-
\left(\Re{(\mathfrak{s})}
+\frac{1}{2}\right)
\int_0^t
\sum_{i=1}^K\mathsf{y}_i^2(s)\mathrm{d}s
  +
 \int_0^t\sum_{i=1}^K\sum_{j\in\mathbb{Z}\setminus\{0,i\}}
\frac{\mathsf{y}_i(s)\mathsf{y}_j^2(s)}{\mathsf{y}_i(s)-\mathsf{y}_j(s)}
\mathrm{d}s
\nonumber\\
&+
\int_0^t\boldsymbol{\gamma}_{HP}(s)
\sum_{i=1}^K\mathsf{y}_i(s)
\mathrm{d}s
+
K\int_0^t
\left(\boldsymbol{\delta}(s)
-
\sum_{j\in\mathbb{Z}\setminus\{0\}}
\mathsf{y}_j^2(s)\right)
\mathrm{d}s.
\end{align}
Here, the
non-coinciding initial condition  $\mathbf{y}$ is such that
$\sum_{i=1}^\infty y_i^2<\infty$,
 and $(\mathsf{MG}_K(t))_{t\ge 0}$ is the corresponding local martingale given by
 \begin{equation*}
 \mathsf{MG}_K(t)=\frac{\sqrt{2}}{2}\int_0^t\sum_{i=1}^K \mathsf{y}_i^2(s)\mathrm{d}\mathsf{w}_i(s), \ \forall t \ge 0.
 \end{equation*}
 Observe that
 \begin{align*}
\sum_{i=1}^K\sum_{j\in\mathbb{Z}\setminus\{0,i\}}
\frac{\mathsf{y}_i(s)\mathsf{y}_j^2(s)}{\mathsf{y}_i(s)-\mathsf{y}_j(s)}
&=
\sum_{\underset{i\neq j}{i,j=1,}}^K\
\frac{\mathsf{y}_i(s)\mathsf{y}_j^2(s)}{\mathsf{y}_i(s)-\mathsf{y}_j(s)}
+
\sum_{i=1}^K\sum_{j\in\mathbb{Z}\setminus\{0,\cdots,K\}}
\frac{\mathsf{y}_i(s)\mathsf{y}_j^2(s)}{\mathsf{y}_i(s)-\mathsf{y}_j(s)}
\\
&=
-\frac{1}{2}\sum_{\underset{i\neq j}{i,j=1,}}^K\
\mathsf{y}_i(s)\mathsf{y}_j(s)
+
\sum_{i=1}^K\sum_{j\in\mathbb{Z}\setminus\{0,\cdots,K\}}
\frac{\mathsf{y}_i(s)\mathsf{y}_j^2(s)}{\mathsf{y}_i(s)-\mathsf{y}_j(s)}, 
\  \ s\in [0,t],
 \end{align*}
and that, 
 \begin{align*}
\sum_{\underset{i\neq j}{i,j=1}}^K\
\mathsf{y}_i(s)\mathsf{y}_j(s) 
\le 
\left( \sum_{i=1}^K \mathsf{y}_i(s)\right)^2
,\  \ 
\sum_{i=1}^K\sum_{j\in\mathbb{Z}\setminus\{0,\cdots,K\}}
\frac{\mathsf{y}_i(s)\mathsf{y}_j^2(s)}{\mathsf{y}_i(s)-\mathsf{y}_j(s)}>0, 
\  \ s\in [0,t].
 \end{align*}
Thus, we have,
\begin{align}\label{delta-sum^2Bound}
K\int_0^t
\left(\boldsymbol{\delta}(s)
-
\sum_{j\in\mathbb{Z}\setminus\{0\}}
\mathsf{y}_j^2(s)\right)
\mathrm{d}s
&\le
      \frac{1}{4}\sum_{i=1}^K\mathsf{y}_i^2(t) 
-\mathsf{MG}_K(t)
+
\left(\Re{(\mathfrak{s})}
+\frac{1}{2}\right)
\int_0^t
\sum_{i=1}^K\mathsf{y}_i^2(s)\mathrm{d}s
 \nonumber\\
& \ \ +
\frac{1}{2}\int_0^t\left(\sum_{i=1}^K \mathsf{y}_i(s)\right)^2
\mathrm{d}s
-
\int_0^t\boldsymbol{\gamma}_{HP}(s)
\sum_{i=1}^K\mathsf{y}_i(s)
\mathrm{d}s.
\end{align}
Note that, all terms on the right hand side above except the last two terms are uniformly bounded as $K\to\infty$. In particular, for $\mathsf{MG}_K$ one can check that the quadratic variation converges and then use the Burkholder-Davis-Gundy inequality \cite{Kallenberg}. As for the last two terms, we claim that after dividing by $K$ they converge to $0$ almost surely. By the Cauchy-Schwartz inequality this would be implied by the almost sure convergence:
 \begin{align*}
    \frac{1}{K^{\frac{1}{2}}}\sup_{s\in [0,t]}\sum_{i=1}^K \mathsf{y}_i(s)
    \overset{K\to\infty}{\longrightarrow} 
    0.
\end{align*}
We now prove this claim. Let $\varepsilon>0$ be arbitrary and by virtue of the convergence of $\sup_{s\in [0,t]}\sum_{i=n+1}^\infty \mathsf{y}_i^2(s)$ to $0$ pick $n$ random large enough so that $\sup_{s\in [0,t]}\sum_{i=n+1}^\infty \mathsf{y}_i^2(s)<\varepsilon^2$. Then, using the Cauchy-Schwarz inequality we can write for all $K>n$
\begin{align*}
    \sup_{s\in [0,t]}\sum_{i=1}^K \mathsf{y}_i(s)
    &= 
    \sup_{s\in [0,t]}\sum_{i=1}^n
    \mathsf{y}_i(s)
    +
    \sup_{s\in [0,t]}\sum_{i=n+1}^K
    \mathsf{y}_i(s)
    \le
    \sup_{s\in [0,t]}\sum_{i=1}^n
    \mathsf{y}_i(s)
    +
    \sqrt{K-n}
    \sqrt{\sup_{s\in [0,t]}\sum_{i=n+1}^K
    \mathsf{y}_i^2(s)}\\
    &\le \sup_{s\in [0,t]} \sum_{i=1}^n \mathsf{y}_i(s)+\varepsilon\sqrt{K-n}.
\end{align*}
Hence, we get that almost surely
\begin{equation*}
\limsup_{K \to \infty} \frac{1}{K^{\frac{1}{2}}}\sup_{s\in [0,t]}\sum_{i=1}^K \mathsf{y}_i(s) <\varepsilon,
\end{equation*}
and since $\varepsilon>0$ was arbitrary this proves the claim. Thus, putting everything together in the bound in 
\eqref{delta-sum^2Bound}, we conclude that almost surely
\begin{align*}
\int_0^t
\left(\boldsymbol{\delta}(s)
-
\sum_{j\in\mathbb{Z}\setminus\{0\}}
\mathsf{y}_j^2(s)\right)
\mathrm{d}s= 
0,
\end{align*}
and hence the SDE reduces to the desired form of Proposition \ref{HPISDE}.
\end{proof}

\subsection{A PDE for Dyson Brownian motion in $1/N$ scaling}

In this final short section we point out a curious fact about the characteristic polynomial of Dyson Brownian motion. It is well-known \cite{Warren,ShkolnikovRamanan,Interlacing,HuaPickrellDiffusions} that the (generalised) Ornstein-Uhlenbeck Dyson Brownian motion with $\mathsf{c}\in \mathbb{R}$ (for $\mathsf{c}=0$ this is standard Dyson Brownian motion while for $\mathsf{c}>0$ it is the standard stationary Ornstein-Uhlenbeck version),
 \begin{align}
		\mathrm{d}\mathfrak{d}_i^{(N)}(t) = 
 \mathrm{d}\mathsf{w}_i(t)  
 -\mathsf{c}\mathfrak{d}^{(N)}_i(t)\mathrm{d}t
 +\sum_{j=1,j\neq i}^N\frac{1}{\mathfrak{d}^{(N)}_i(t)-\mathfrak{d}^{(N)}_j(t)}
 \mathrm{d}t, \ \  i\in \llbracket 1, N \rrbracket,
\end{align}
is consistent i.e. the semigroups are intertwined for different $N$ in the sense of Section \ref{SubsectionConvPathSpace}, see in particular \eqref{Intertwining} from Proposition \ref{PropConsistency}. Moreover, the relevant regularity properties such as the Feller property, continuity of paths and smooth symmetric function generator core are known \cite{DunklProcesses}. Hence, using the theory of \cite{AM24} after embedding into $\Upsilon$ one has convergence in $\mathrm{C}(\mathbb{R}_+;\Upsilon)$. Let us then define
\begin{equation*}
\mathfrak{D}_{t;N}(z)=\prod_{i=1}^N\left(1-\frac{\mathfrak{d}_i^{(N)}(t)z}{N}\right).
\end{equation*}
Note that, dividing by $N$ is too big of a scaling in order to see non-trivial limiting stochastic behaviour for $\mathfrak{D}_{t;N}$, as in for example the edge scaling case towards the Airy line ensemble \cite{HuangZhang}. We have the following result. 

\begin{prop} Let $\mathsf{c}\in \mathbb{R}$. Suppose $\mathfrak{D}_{0;N} \overset{N \to \infty}{\longrightarrow} \mathfrak{D}_0 \in \mathfrak{LP}$. Then, as $N \to \infty$, $\mathfrak{D}_{\bullet;N} \overset{\mathrm{d}}{\longrightarrow}\mathfrak{D}_{\bullet}$ in $\mathrm{C}(\mathbb{R}_+;\mathfrak{LP})$ with the (deterministic) evolution $t\mapsto\mathfrak{D}_t$ solving the equation
\begin{equation}\label{1stOrderPDE}
\partial_t \mathfrak{D}_t(z)=-\frac{z^2}{2}\mathfrak{D}_t(z)-\mathsf{c}z\partial_z\mathfrak{D}_t(z).
\end{equation}
Moreover, for $\mathfrak{D}_0\in \mathfrak{LP}$, the unique solution to this equation is given by,
    \begin{align*}
        \mathfrak{D}_t(z)
        =\mathfrak{D}_0\left(z\mathrm{e}^{-\mathsf{c}t}\right)
        \exp\left(\frac{-z^2}{4\mathsf{c}}
        \left(1-\mathrm{e}^{-2\mathsf{c}t}\right)\right),
    \end{align*}
which for $\mathsf{c}=0$ simplifies to $\mathfrak{D}_t(z)=\mathfrak{D}_0(z)\mathrm{e}^{-tz^2/2}$.
\end{prop}

\begin{proof}
The convergence statement for $\mathfrak{D}_{\bullet;N}$ follows by virtue of consistency while the fact that $\mathfrak{D}_{\bullet}$ solves \eqref{1stOrderPDE} follows by deriving an equation for $\mathfrak{D}_{\bullet;N}$ as before and taking the limit. It is tedious but elementary so we omit the details. The linear first order pde \eqref{1stOrderPDE} can be solved explicitly using the method of characteristics. Alternatively,  one can solve at the level of the $\mathfrak{LP}$-class representation of the function with parameters $((\mathsf{x}_i^+(\bullet))_{i\in \mathbb{N}},(\mathsf{x}_i^-(\bullet))_{i\in \mathbb{N}},\boldsymbol{\gamma}(\bullet),\boldsymbol{\delta}(\bullet))$ as follows (one can also prove convergence in this way, see \cite{HuaPickrellDiffusions} where this was discussed):
\begin{equation*}
\mathsf{x}_i^{\pm}(t)=\mathsf{x}^\pm_i(0)\mathrm{e}^{-\mathsf{c}t}, \ \boldsymbol{\gamma}(t)=\boldsymbol{\gamma}(0)\mathrm{e}^{-\mathsf{c}t}, \ \boldsymbol{\delta}(t)=\frac{1}{2\mathsf{c}}\left(1-\mathrm{e}^{-2\mathsf{c}t}\right)+\boldsymbol{\delta}(0)\mathrm{e}^{-2\mathsf{c}t}
\end{equation*}
and recombine to get $\mathfrak{D}_t$.
\end{proof}

\bibliographystyle{acm}
\bibliography{References}

\bigskip 

\noindent{\sc School of Mathematics, University of Edinburgh, James Clerk Maxwell Building, Peter Guthrie Tait Rd, Edinburgh EH9 3FD, U.K.}\newline
\href{mailto:theo.assiotis@ed.ac.uk}{\small theo.assiotis@ed.ac.uk}

\bigskip

\noindent{\sc School of Mathematics, University of Edinburgh, James Clerk Maxwell Building, Peter Guthrie Tait Rd, Edinburgh EH9 3FD, U.K.}\newline
\href{mailto:Z.S.Mirsajjadi@sms.ed.ac.uk}{\small Z.S.Mirsajjadi@sms.ed.ac.uk}

\end{document}